\documentclass[a4paper,12pt,oneside,reqno]{amsart}
\usepackage{amsfonts,amsmath, amssymb,amsthm,amscd,mathtools}
\usepackage[usenames,dvipsnames]{color}
\usepackage{bm}
\usepackage{mathrsfs}
\usepackage[mpexclude,DIV13]{typearea}
\usepackage{verbatim}
\usepackage{hyperref}
\usepackage{graphicx}
\usepackage{multicol}
\usepackage{latexsym}
\usepackage{lscape}
\usepackage{epsfig}
\usepackage{hyperref}
\usepackage[margin=1in]{geometry}
\usepackage{tikz,bbm}
\usepackage{pgfplots,xspace}
\usetikzlibrary{decorations.pathmorphing, decorations.pathreplacing, decorations.shapes} 

\newcommand{\C}{\mathbb{C}}
\newcommand{\Z}{\mathbb{Z}}
\newcommand{\N}{\mathbb{N}}

\newcommand{\wt}{\mathrm{wt}}
\newcommand{\fwt}{\mathrm{wt}^\mathrm{fus}}
\newcommand{\ufwt}{\mathrm{wt}^\mathrm{unfus}}

\newcommand{\numrows}{r}
\newcommand{\numcols}{N}
\newcommand{\lam}{\lambda}

\renewcommand{\v}{\mathbf{v}}
\newcommand{\z}{\mathbf{z}}
\renewcommand{\c}{\mathbf{c}}
\renewcommand{\d}{\mathbf{d}}
\newcommand{\longelt}{w_0}
\newcommand{\wi}[1]{w_{#1}}
\newcommand{\wc}{w_c}
\newcommand{\wdd}{w_d}
\renewcommand{\S}{\mathcal{S}}
\newcommand{\Si}[1]{\mathfrak{S}_{#1}}
\newcommand{\W}{S_r}
\newcommand{\Zbar}{\tilde{Z}}

\newtheorem{theorem}{Theorem}[section]

\newtheorem*{theorem*}{Theorem}
\newtheorem{lemma}[theorem]{Lemma}
\newtheorem{definition}[theorem]{Definition}
\newtheorem{proposition}[theorem]{Proposition}

\newtheorem{corollary}[theorem]{Corollary}

\definecolor{cb-white}{RGB}{255,255,255}
\definecolor{cb-light-yellow}{RGB}{238,204,102}
\definecolor{cb-light-red}{RGB}{238,153,170}
\definecolor{cb-light-blue}{RGB}{102,153,204}
\definecolor{cb-dark-yellow}{RGB}{153,119,0}
\definecolor{cb-dark-red}{RGB}{153,68,85}
\definecolor{cb-dark-blue}{RGB}{0,68,136}

\definecolor{cb-blue-1}{RGB}{0,68,136}
\definecolor{cb-purple-1}{RGB}{170,51,119}
\definecolor{cb-cyan-1}{RGB}{102,204,238}
\definecolor{cb-green-1}{RGB}{34,136,51}
\definecolor{cb-yellow-1}{RGB}{204,187,68}
\definecolor{cb-red-1}{RGB}{238,102,119}
\definecolor{cb-gray-1}{RGB}{187,187,187}

\newcommand{\rvertexdolcrosstwo}{\tikz[anchor=base, baseline=-2pt, scale=1]{		
		\draw[cb-dark-blue, line width=.75mm, fill=white, dash pattern=on 4pt off 3pt] (.7,.7)--(0,0);
		\draw[cb-dark-red, line width=.75mm, fill=white, dash pattern=on 4pt off 3pt] (-.7,.7)--(0,0);
		\draw[cb-dark-blue, line width=.75mm, fill=white, dash pattern=on 4pt off 3pt] (0,0)--(-.7,-.7);
		\draw[cb-dark-red, line width=.75mm, fill=white, dash pattern=on 4pt off 3pt] (0,0)--(.7,-.7);
		
		\draw[cb-dark-blue, fill=cb-dark-blue, line width=.75mm] (-.7,-.7) circle (.1);
		\draw[cb-dark-blue, fill=cb-dark-blue, line width=.75mm] (.7,.7) circle (.1);
		\draw[cb-dark-red, fill=cb-dark-red, line width=.75mm] (-.7,.7) circle (.1);
		\draw[cb-dark-red, fill=cb-dark-red, line width=.75mm] (.7,-.7) circle (.1);

        \path[fill=white] (0,0) circle (.3);
        \node at (0,0) {$\footnotesize z_{i},z_{i+1}$};
	}\xspace}

\newcommand{\rvertexcolsametwo}{\tikz[anchor=base, baseline=-2pt, scale=1]{
		\draw[cb-light-red, line width=.75mm, fill=white, line cap=round] (.7,.7)--(0,0);
		\draw[cb-light-red, line width=.75mm, fill=white, line cap=round] (-.7,.7)--(0,0);
		\draw[cb-light-blue, line width=.75mm, fill=white, line cap=round] (0,0)--(-.7,-.7);
		\draw[cb-light-blue, line width=.75mm, fill=white, line cap=round] (0,0)--(.7,-.7);
		
		\draw[cb-light-blue, fill=cb-light-blue, line width=.75mm] (-.7,-.7) circle (.1);
		\draw[cb-light-red, fill=cb-light-red, line width=.75mm] (.7,.7) circle (.1);
		\draw[cb-light-red, fill=cb-light-red, line width=.75mm] (-.7,.7) circle (.1);
		\draw[cb-light-blue, fill=cb-light-blue, line width=.75mm] (.7,-.7) circle (.1);

        \path[fill=white] (0,0) circle (.3);
        \node at (0,0) {$\footnotesize z_i,z_{i+1}$};
		}\xspace}
	
\newcommand{\rvertexcolcrosstwo}{\tikz[anchor=base, baseline=-2pt, scale=1]{		
		\draw[cb-light-red, line width=.75mm, fill=white, line cap=round] (.7,.7)--(0,0);
		\draw[cb-light-blue, line width=.75mm, fill=white, line cap=round] (-.7,.7)--(0,0);
		\draw[cb-light-red, line width=.75mm, fill=white, line cap=round] (0,0)--(-.7,-.7);
		\draw[cb-light-blue, line width=.75mm, fill=white, line cap=round] (0,0)--(.7,-.7);
		
		\draw[cb-light-red, fill=cb-light-red, line width=.75mm] (-.7,-.7) circle (.1);
		\draw[cb-light-red, fill=cb-light-red, line width=.75mm] (.7,.7) circle (.1);
		\draw[cb-light-blue, fill=cb-light-blue, line width=.75mm] (-.7,.7) circle (.1);
		\draw[cb-light-blue, fill=cb-light-blue, line width=.75mm] (.7,-.7) circle (.1);

        \path[fill=white] (0,0) circle (.3);
        \node at (0,0) {$\footnotesize z_i,z_{i+1}$};
	}\xspace}

\newcommand{\rvertexdolsametwo}{\tikz[anchor=base, baseline=-2pt, scale=1]{		
		\draw[cb-dark-red, line width=.75mm, fill=white, dash pattern=on 4pt off 3pt] (.7,.7)--(0,0);
		\draw[cb-dark-red, line width=.75mm, fill=white, dash pattern=on 4pt off 3pt] (-.7,.7)--(0,0);
		\draw[cb-dark-blue, line width=.75mm, fill=white, dash pattern=on 4pt off 3pt] (0,0)--(-.7,-.7);
		\draw[cb-dark-blue, line width=.75mm, fill=white, dash pattern=on 4pt off 3pt] (0,0)--(.7,-.7);
		
		\draw[cb-dark-blue, fill=cb-dark-blue, line width=.75mm] (-.7,-.7) circle (.1);
		\draw[cb-dark-red, fill=cb-dark-red, line width=.75mm] (.7,.7) circle (.1);
		\draw[cb-dark-red, fill=cb-dark-red, line width=.75mm] (-.7,.7) circle (.1);
		\draw[cb-dark-blue, fill=cb-dark-blue, line width=.75mm] (.7,-.7) circle (.1);

        \path[fill=white] (0,0) circle (.3);
        \node at (0,0) {$\footnotesize z_{i},z_{i+1}$};
	}\xspace}

\newcommand{\gamlittle}{\tikz[baseline=-2pt, scale=1]{		
		\draw[line width=.25mm] (1,0)--(-1,0);
		\draw[line width=.25mm] (0,1)--(0,-1);

        \path[fill=blue!25!white] (-1,0) circle (.3);
        \path[fill=blue!25!white] (1,0) circle (.3);
        \path[fill=blue!25!white] (0,-1) circle (.3);
        \path[fill=blue!25!white] (0,1) circle (.3);
		
		\node at (-1,0) {$a$};
		\node at (0,1) {$b$};
		\node at (1,0) {$c$};
		\node at (0,-1) {$d$};
	}\xspace}

\newcommand{\gambig}{\tikz[baseline=-2pt, scale=1]{		
		\draw[line width=.25mm] (1,0)--(-1,0);
		\draw[line width=.25mm] (0,1)--(0,-1);
		
        \path[fill=blue!25!white] (-1,0) circle (.3);
        \path[fill=blue!25!white] (1,0) circle (.3);
        \path[fill=blue!25!white] (0,-1) circle (.3);
        \path[fill=blue!25!white] (0,1) circle (.3);
		
		\node at (-1,0) {$A$};
		\node at (0,1) {$B$};
		\node at (1,0) {$C$};
		\node at (0,-1) {$D$};
	}\xspace}

\begin{document}
\begin{abstract}
    Recent papers in solvable lattice models emphasize models where states can be visualized as colored paths through the lattice. We define a bosonic model in which there are two types of colors, one whose paths move down and to the right, the other whose paths move down and to the left. 
    Depending on their boundary data, systems may have no states, exactly one state, or many states. We prove that these cases depend on a criterion involving two permutations extracted from the boundary data and their Bruhat order. This classification also helps us to characterize the partition functions of our systems, a question at the heart of the study of solvable lattice models. Using the solvability of the model, we derive a four-term recurrence relation on the partition function. Together with the classification of systems by number of states which serves as a base case for the recursion, the recursion completely characterizes the partition function of systems. We also show a color merging property relating the bicolored bosonic models to colored and uncolored bosonic models, and correspondence with Gelfand-Tsetlin patterns.
\end{abstract}

\title{Bicolored bosonic solvable lattice models}
\author[T. Blum]{Talia Blum}
\address{T. Blum,
Stanford University,
Department of Mathematics,
Stanford, CA 94305}
\email{taliab@stanford.edu}

\maketitle

\setcounter{tocdepth}{1}
\tableofcontents
\hypersetup{linktocpage}

\section{Introduction}
In solvable lattice models, states can often be described as sets of paths through a lattice. Since Borodin and Wheeler \cite{BW}, these are often described as \textit{colored} paths. The models in this paper are unusual because there are two sets of colors giving two sets of paths that interact in a coordinated way.

We will define a bicolored bosonic model lattice model, which differs from previously studied colored models in that it includes two types of bosonic colors which each behave differently. This construction is reminiscent of the supersymmetric models studied by Brubaker, Buciumas, Bump, and Friedberg in \cite{BBBG2}, with a key distinction that their models are fermionic, rather than bosonic.

We consider a lattice model constructed on a square grid. The edges are decorated with spins among two types: we call them \textit{colors} and \textit{dolors} (dual colors). A \textit{state} of the model is an admissible assignment of spins to every edge. An example of an admissible state of the model is shown in Figure~\ref{fig:state_example_intro}. To intuitively describe an admissible state, we can think of the colors and dolors as particles tracing out paths through the grid. Colors trace out down-right paths, while dolors travel down and to the left. On vertical edges, particles must travel in pairs, each pair consisting of one color and one dolor. On horizontal edges the particles travel alone. 
The model is \textit{bosonic} in the sense that many pairs of particles may travel along the same vertical edge. In admissible states of the model, the colors trace out paths travelling down and to the right, and dolors trace out paths travelling down and to the left. To each state we may also associate a \textit{Boltzmann weight}.

\begin{figure}[ht!]
    \centering
    \input{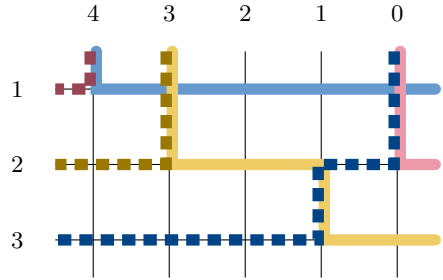}
    \caption{An example of a state of the (fused) model. Colors are represented by solid lines, which form paths travelling down and to the right. Dolors are represented by dashed lines, which form paths travelling down and to the left. On vertical edges, pairs of a color and dolor must travel together, while on horizontal edges exactly one particle occupies an edge.}
    \label{fig:state_example_intro}
\end{figure}

A \textit{system} is a collection of states with the same boundary conditions. The \textit{partition function} the sum of the Boltzmann weights of the states in a system. The partition function is a central object in the study of solvable lattice models, as it is often a rich source of connections with other areas of mathematics. Partition functions of certain colored lattice models are symmetric functions \cite{BW} and matrix coefficients of $p$-adic representations \cite{BBBG, BN}. 

A main result of this paper is the following result on the partition function of systems of the bicolored bosonic model.

\begin{theorem}\label{thm:intro_recurrence_determines_pt}
    The partition function of a system of the bicolored bosonic model satisfies a four-term recurrence relation.
    Together with characterization of the partition functions of systems with exactly one state and systems with no states, the recurrence completely determines the partition function.
\end{theorem}

This result is a powerful characterization of the partition function. There are two important underlying results, which we will now discuss at greater length.

\begin{enumerate}
    \item The four-term recurrence relation on the partition function.
    \item Knowledge of systems with exactly one state and with no states.
\end{enumerate}

The first of these is the recurrence relation on the partition function. For this result, we use that the model is \textit{solvable}. Solvable models satisfy a remarkable relation called the \textit{Yang-Baxter equation}. This tool is often a key to studying the partition function. Indeed, we will use the Yang-Baxter equation to show that the  to show that the partition function of the bicolored bosonic model satisfies the aforementioned four-term recursion. In \cite{BN}, the Yang-Baxter equation is similarly used to show a three-term recursion for the partition function of a colored bosonic lattice model; it is expressed in terms of Hecke module actions of Demazure-Lusztig operators that also come up in representation theory of of $p$-adic groups \cite{BruBuLi, CheMa, CheMa2, Ion} and $K$-theory of flag varieties \cite{Lusztig}. In Section~\ref{sec:recurrence}, we express the four-term recurrence as divided difference operators on the partition function. It is likely that these operators in this paper similarly describe actions of a Hecke algebra.

For (2), we mean the following result, which classifies systems by number of states using only certain boundary data. 

\begin{theorem}\label{thm:intro_mono}
    Consider a system with permutations $\wc , \wdd  \in S_r$ determined by the orderings of the colors and dolors respectively on the boundary, and $\lam$ a partition specifying the positions of the bosons on the top boundary.
    \begin{itemize}
        \item If $\wdd \not\leq \wc$ in the Bruhat order, then the system has no states.
        \item If $\wdd = \wc$, then the system has exactly one state.
        \item If $\wdd \leq \wc$ and $\lam$ is sufficiently dominant, then the system has many states.
    \end{itemize}
\end{theorem}

\begin{figure}[ht!]
    \centering
    \input{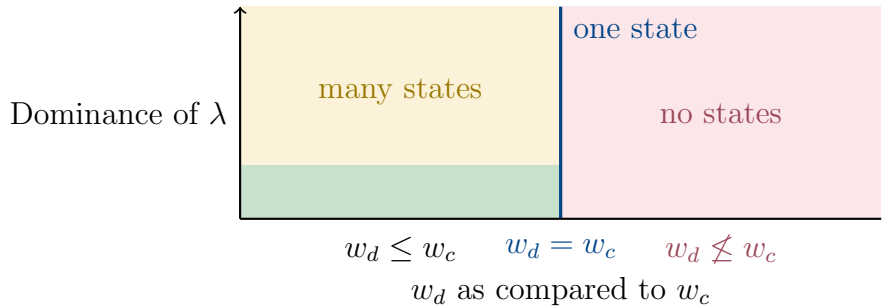}
    \caption{A phase diagram summarizing the number of states of a system based on its boundary data. Note that when $\wdd \not < \wc$, the number of states does not depend on $\lam$ at all. Once $\lam$ is sufficiently dominant, the condition that $\wdd \leq \wc$ is again enough to say that the system has at least one state. In the green region (bottom left), where $\lam$ is not sufficiently dominant, a system may have no states, one state, or many states.}
    \label{fig:mono_states_phases}
\end{figure}

Using the above result we are able to discern which systems have no states, one state, or many states based on a Bruhat order relation on their boundary data. For the first two conditions, it is powerful that we may identify systems with one or no states using only the Bruhat order condition on some of the boundary data. In the third case, systems with $\wdd \leq \wc$ may have any number of states, provided that $\lam$ is sufficiently dominant.
Aside from being useful in characterizing the partition function, this is an interesting result on the combinatorics of systems of the model. 

To prove Theorem~\ref{thm:intro_recurrence_determines_pt} we show that (1) and (2) together determine the partition function of any system. We are able to easily determine the partition functions of the two types of special systems in Theorem~\ref{thm:intro_mono}: for a system with no states, its partition function is 0, and for a monostatic system, the partition function is the weight of its sole state. So it remains to show that we can compute the partition function for systems with $\wdd \leq \wc$, which may have many states. In Section~\ref{sec:recurrence} we conclude the proof of Theorem~\ref{thm:intro_recurrence_determines_pt} by showing that for such a system $\Si{}$ with $\wdd \leq \wc$, we may use the four-term recurrence relation to write $Z(\Si{})$ in terms of partition functions of systems descending toward those with $\wdd \not < \wc$. The partition functions of monostatic systems and systems with no states serve as base cases for the recursion.

There are similar results to Theorem~\ref{thm:intro_recurrence_determines_pt} in the literature on solvable lattice models. It is a standard application of the Yang-Baxter equation to prove a recurrence relation on the partition function. For most previously studied models this recursion often has only two or three terms: generally uncolored models have two terms \cite{BBF,I,Z} and colored models have three terms \cite{BN, BW}. In \cite{BBBG2}, there is a four-term relation on the partition function. They show that this recursion determines the partition function by identifying ``ground state'' systems, which comprise a subset of the systems which have exactly one state in their model. Theorem~\ref{thm:intro_mono} provides a more general result classifying systems exactly by quantity of states.

We will also show a process of \textit{color merging} connecting our bicolored bosonic models to previously studied colored bosonic models \cite{BN, BW} and uncolored bosonic models \cite{Ko}. This process has been called \textit{local lifting} by \cite{BN} and \textit{color-blindness} by \cite{BW}. By locally merging the dolors, we are able to recover the weights of the colored models in \cite{BN}. This local property has global implications on the partition functions of systems, implying connections with matrix coefficients of $p$-adic groups. By way of the color merging process on the colored bosonic models to uncolored models in \cite{BN}, we may also draw connections with Hall-Littlewood polynomials \cite{Ko, M}.

It is well known that states of solvable lattice models correspond to combinatorial objects such as semi-standard Young tableaux, Gelfand-Tsetlin patterns, and tilings of the plane. We will define certain 2-colored Gelfand-Tsetlin patterns and show that they correspond to states of the bicolored bosonic model.

A summary of the main results is below.

\begin{enumerate}
    \item \textit{The bicolored bosonic model.} We define two equivalent versions of the model, referred to as \textit{unfused} and \textit{fused}. We apply a process of vertex fusion to create these weights for the fused model from the unfused model.
    \item \textit{The model is solvable.} This means that it satisfies the Yang-Baxter equation.
    \item \textit{Classification of systems with zero and one state.} We prove a condition on the boundary conditions defining a system which determines when a system has no states or has exactly one state. This condition compares under the Bruhat order two permutations which define the boundary of a system. 
    \item \textit{The partition function satisfies a four-term recursion.} This recursion plus knowledge of special systems determines the recurrence relation. We also express this relation in terms of divided difference operators on the partition function. This will also likely be a source of connections to matrix coefficients of $p$-adic representations, similar to \cite{BN}. 
    \item \textit{Color merging.} We show that locally the weights of these bicolored bosonic models are generalizations of colored and uncolored bosonic models. This local property has global implications, showing that our model has connections to matrix coefficients of $p$-adic groups and Hall-Littlewood polynomials as in \cite{BN}, for example.
    \item \textit{Correspondence with Gelfand-Tsetlin patterns.} We show that states of the model are in bijection with certain 2-colored Gelfand-Tsetlin patterns. 
\end{enumerate}

\subsection{Organization} 
The paper is organized as follows. First, we define the bicolored bosonic lattice model. In Section~\ref{sec:monostatic_systems} we show combinatorial conditions involving the Bruhat order for when a system has no states, exactly one state, or many states. In Section~\ref{sec:solvable}, we introduce the $R$-vertex and show that the Boltzmann weights for the model satisfy the Yang-Baxter equation, proving that the model is \emph{solvable}. In Section~\ref{sec:recurrence}, we use the solvability of the model to compute the partition function via a recurrence relation. In Section~\ref{sec:llp}, we show that these models are related to colored (as opposed to bicolored) bosonic lattice models of \cite{BN} and uncolored bosonic lattice models. Finally in Section~\ref{sec:gt_patterns} we show a bijection between states of the lattice model and 2-colored Gelfand-Tsetlin patterns.

\subsection{Acknowledgements}
The author thanks Daniel Bump for many valuable conversations and suggestions on this work. The author is also grateful to Andy Hardt and Slava Naprienko for helpful conversations. This work was funded by the NSF Graduate Research Fellowship.

\section{The bicolored bosonic model}
In this section we describe the vertex model, and its systems and states. The unfused model is a building block to ultimately construct the fused model, which we will be working with for the rest of the paper unless otherwise specified. Then we will describe the boundary conditions which describe the systems we will work with for the rest of the paper.

The vertex models we consider in this paper will have two types of particles: \emph{colors} which trace out down-right paths, and \emph{dolors} which trace out down-left paths. We can think of the \emph{dolors} as `dual colors,' or almost as antiparticles for the colors: on vertical edges, pairs of one color and one dolor travel together, and if the color turns right on a horizontal edge, its paired dolor must turn left.

Let $\c = \{c_1,\dots,c_r\}$ be the set of colors and $\d = \{d_1,\dots,d_s\}$ the set of dolors. We refer to pairs $(c,d)\in\c\times\d$ as \emph{bosons}.

\subsection{The unfused model} 
An unfused vertex is associated with three parameters: a color $c \in \c$, a dolor $d \in \d$, and a spectral parameter $z \in \C$. The admissible unfused vertices and their Boltzmann weights are shown in Figure~\ref{fig:monochrome_weights}. On a vertical edge, pairs of particles of types $c,d$ must travel together. The model is \emph{bosonic}, so any number $n \in \N$ of $(c,d)$ pairs may travel along a vertical edge. The set of possible spins for a vertical edge therefore is
\[
\Sigma_v^{(c,d)}= \{0,1,2,\cdots\} = \N.
\]
Horizontal edges are occupied by exactly one particle, so the spinset for a horizontal edge is 
\[
\Sigma_h^{\textrm{unfus}} = \c \cup \d.
\]

\begin{figure}[ht!]
    \centering
    \input{figures/monochrome_weights.tex}
    \caption{Admissible unfused vertices. Each unfused vertex is associated with a color $c\in \c$ (solid yellow), a dolor $d \in \d$ (dotted yellow), and a spectral parameter $z\in \C$. The weights also depend on another global parameter $t\in \C$. The number $n\in \Z_{\geq 0}$ represents the number of bosons of type $(c,d)$ which are travelling along a vertical edge. In the images, dolors are represented by dotted lines ($x\in\d$) and colors by solid lines ($a\in\c$). Any vertex configurations not pictured has Boltzmann weight 0.}
    \label{fig:monochrome_weights}
\end{figure}

\subsection{The fused model}\label{sec:fused_model}
An admissible fused vertex results from the `fusion' of several admissible unfused vertices. We describe the fusion process in detail below.

Suppose $\c$ is the set of $r$ colors and $\d$ is the set of $s$ dolors. The set $\c\times\d$ identifies types of \emph{bosons}, which are pairs of one color and one dolor. We impose an ordering on the bosons.

\begin{definition}[Monochrome ordering of bosons]\label{def:mono_ordering}
    We define $(c_1,d_s)$ to be the minimal boson. Given $(c_i, d_j) \in \c\times\d$ its successor $(c_i, d_j)+1$ is
    \begin{equation}\label{eq:mono_ordering}
    	(c_i, d_j) + 1 = 
    	\begin{cases}
    		(c_{i+1},d_j) & \textrm{if }i< r, \\
    		(c_1,d_{j-1}) & \textrm{if }i = r, j > 1, \\
    		(c_1,d_s) & \textrm{if }i = r, j = 1.
    	\end{cases}
    \end{equation}

    Equivalently, given two bosons $(c,d), (c',d') \in \c\times\d$ we say that $(c,d) < (c',d')$ if $d > d'$, or both $d = d'$ and $c < c'$.
\end{definition}
We may denote a dolor-color pair $(c,d)$ as $cd$ when possible, and its successor as $(c,d)+1$ or $cd + 1$. We also may write $(c,d)_i$ or $(cd)_i$; we define the minimal element $(c_1,d_s) = (cd)_1$.

\begin{figure}[ht!]
	\input{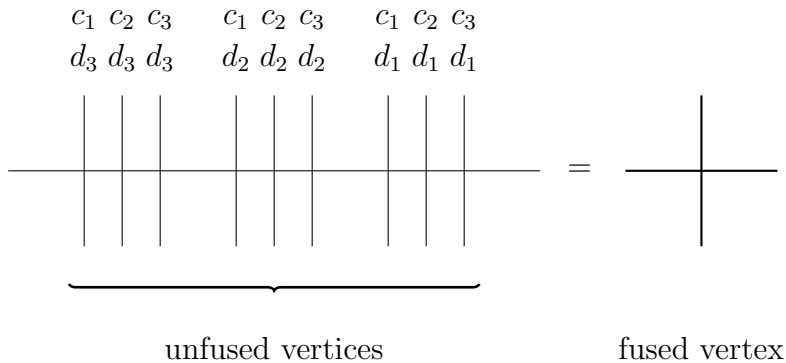}
	\caption{Vertex fusion replaces many unfused vertices with one equivalent fused vertex. Vertical edges on the left hand side are labelled by bosons according to the monochrome ordering in Definition~\ref{def:mono_ordering}. The weight of a fused vertex is the product of the weights of those unfused vertices. Here, $r=s=3$.}
	\label{fig:vertex_fusion}
\end{figure}

A \emph{fused vertex} replaces $rs$ unfused vertices, as in the example in Figure~\ref{fig:vertex_fusion}. The spinset for horizontal edges in a fused system is unchanged from the unfused system 
\[\Sigma_h^{\textrm{fus}} = \Sigma_h^{\textrm{unfus}} = {\bf c} \cup {\bf d}.\] 
The spinset for vertical edges in the fused system is the product of the spinsets for the correseponding monochrome vertical edges,
\[
\Sigma_v^\text{fus} = \prod_{(c,d) \in \c\times\d} \Sigma_v^{(c,d)} = \N^{rs},
\]
where each copy of $\N$ corresponds to the number of bosons of each type.

We say the labelling of edges around a fused vertex is admissible if there is an admissible labelling of the unfused vertices. The Boltzmann weight of a fused vertex is the product of the weights of the monochrome vertices in this decomposition.

\begin{equation}\label{eq:fused_weights}
\fwt(\v) = \prod_{(c,d) \in \c \times \d } \ufwt(\v_{(c,d)}).
\end{equation}

\subsection{Boundary conditions for a system}
In this subsection we define systems and states of the model. A system is an ensemble of states, in our case the ensemble of states sharing the same boundary conditions. A state of the model is constructed on a square grid with a fixed number of $\numrows$ rows and $\numcols$ columns. At each intersection of a row and column, there is a vertex, and each vertex is adjacent to four edges. Edges which are adjacent to two vertices are called \emph{internal edges}, and edges which are adjacent to only one vertex are called \emph{boundary edges}.

An \emph{(admissible) state} of the model is an assignment of spins to all edges, such that the configuration of edges around each vertex is an admissible (fused) vertex. A \emph{boundary condition} is a fixed labelling of the boundary edges.
A \emph{system} is the collection of states sharing the same boundary conditions and the weights of those states.

We will study systems with the following boundary conditions throughout this paper. Let $\numrows,\numcols \in \N$ and take $\c = \{c_1,\dots,c_r\}$ and $\d = \{d_1,\cdots,d_r\}$.
Let $\lam$ be a partition of length $\numrows$ and $\lam_1 < \numcols$. Let $\wi{1},\wi{2},\wi{3},\wi{4} \in S_\numrows$, and denote the long element by $\longelt \in S_\numrows$. Let $\z = (z_1,\dots, z_\numrows) \in \C^\numrows$ be the spectral parameter. 

\begin{definition}\label{def:boundary_conditions}
    The \emph{system} $\S(\lam, \wi{1},\wi{2},\wi{3},\wi{4}, \z)$ defines the collection of states with the following boundary conditions. 
    \begin{itemize}
        \item The grid has $\numrows$ rows and $\numcols$ columns. The rows are labelled $1,\dots,\numrows$ from top to bottom and the columns are labelled $\numcols-1,\dots,0$ from left to right.
        \item The right boundary edges are occupied by colors $\wi{1}\c$.
        \item The left boundary edges are occupied by dolors $\wi{2}\d$. 
        \item The top boundary edges are occupied by the pairs $(\wi{3}\c, \longelt\wi{4}\d)_i$ in the columns $\lam_i$.
        \item The bottom boundary edges are unoccupied.
        \item The spectral parameter $\z$ labels the rows.
    \end{itemize}
\end{definition}

See Figure~\ref{fig:fused_state_ex} for an example of both states of a system of the model.

\begin{figure}[ht!]
	\input{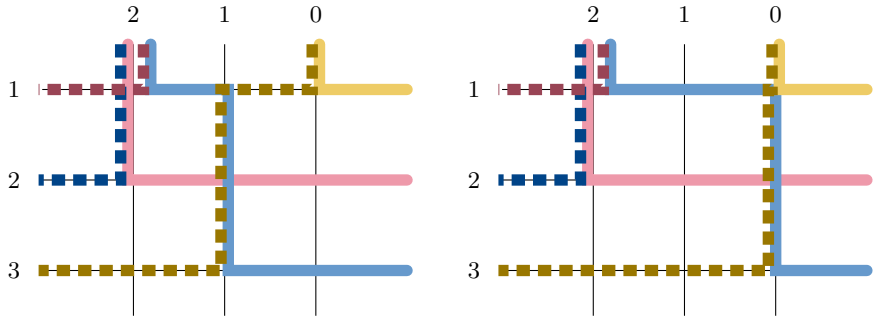}
	\caption{Both states of the system $\S(\lam,\wi{1},\wi{2},\wi{3},\wi{4})$ with boundary conditions given by $\lam =(2,2,0)$, $\wi{1} = (123)$, $\wi{2} = ()$, $\wi{3} = (12)$, and $\wi{4} = (123)$.}
	\label{fig:fused_state_ex}
\end{figure}

\subsection{The partition function} The partition function of a system is a central object in the study of solvable lattice models. In later sections we will see different ways to compute it for special systems, and then for a general system. We will define it in terms of the Boltzmann weights.

The \emph{Boltzmann weight of a state $s$} of the lattice model is the product of the Boltzmann weights of its vertices,
\[
\wt(s) = \prod_{\v \in s} \fwt(\v),
\]
where we recall that we compute the fused weight of a vertex from the unfused weights in Figure~\ref{fig:monochrome_weights} and the fusion process described in Figure~\ref{fig:vertex_fusion} and Equation~\eqref{eq:fused_weights}.

\begin{definition}\label{def:partition_function}
The \emph{partition function} of a system $\Si{}$ is the sum of the weights of its states,
\[
Z(\Si{}) = \sum_{s\in\Si{}} \wt(s).
\]
\end{definition}

\section{Combinatorial conditions for number of states}\label{sec:monostatic_systems}

In this section we will classify systems according to number of states by a condition on their boundary data alone. This result is interesting as a characterization of the combinatorics of the lattice model. It will also show to be extremely useful for us later in the paper, as an important piece in the puzzle of being able to compute partition functions of systems using a recurrence relation instead of the Boltzmann weights.

Recall that we define a system by the boundary conditions described in Definition~\ref{def:boundary_conditions}. The main result of this section is the following theorem.

\begin{theorem}
    Consider a system $\Si{} = \S(\lam,\wi{1},\wi{2},\wi{3},\wi{4},\z)$. Define the permutations $\wc = \wi{3}\wi{1}^{-1}$ and $\wdd = \longelt\wi{4}\wi{2}^{-1}$. Then the following are true.
    \begin{itemize}
        \item If $\wdd \not\leq \wc$ in the Bruhat order, then $\Si{}$ has no states.
        \item If $\wdd = \wc$, then $\Si{}$ has exactly one state.
        \item If $\wdd \leq \wc$ and $\lam$ is \emph{sufficiently dominant}, then $\Si{}$ has at least one state.
    \end{itemize}
By \textit{sufficiently dominant}, we mean that $\lam_i - \lam_{i+1} \geq N$, for sufficiently large $N$.
\end{theorem}

We will first characterize the \emph{monostatic systems}, that is, the systems with exactly one state. Then we will do the same for systems with no states. We will use the notation $\wc, \wdd$ as in the theorem, so we redefine these permutations below to reference throughout the section.

\begin{definition}\label{def:wc}
    For a system $\Si{}$, define the permutations $\wc(\Si{}) = \wi{3}\wi{1}^{-1}$ and $\wdd(\Si{}) = \longelt\wi{4}\wi{2}^{-1}$. When the system $\Si{}$ is unambiguous, we may simply write $\wc, \wdd$.
\end{definition}

\subsection{Systems with exactly one state}
\emph{Monostatic systems} are the systems with exactly one state. There are two reasons why a system may have exactly one state: the number of states may be restricted by either: 
\begin{enumerate}
    \item the permutations $\wi{1},\wi{2},\wi{3},\wi{4}$ defining the boundary, or \item the dominance of the permutation $\lam$.
\end{enumerate}
In this section we will characterize monostatic systems of this first type, an example of which is shown in Figure~\ref{fig:monostatic_example}. In fact we will show that in this first case, $\lam$ has no effect on the number of states of the system. On the other hand, an example of a system which has only one state due to a restriction by $\lam$ is shown in Figure~\ref{fig:fused_state_ex_mono}, as compared to the system with multiple states in Figure~\ref{fig:fused_state_ex}. We will discuss the conditions on the dominance of $\lam$ later.

\begin{figure}[ht!]
	\input{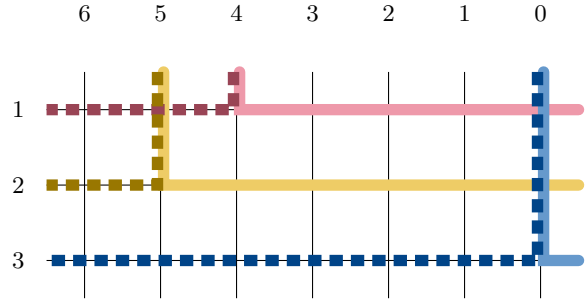}
	\caption{This system has boundary conditions $\lam = (5,4,0)$, $\wi{1} = (23)$, $\wi{2} = (23)$, $\wi{3} = (123)$, $\wi{4} = (12)$, and we can compute $\wdd = (12)$, $\wc=(12)$. The system has exactly one state.}
	\label{fig:monostatic_example}
\end{figure}

\begin{figure}[ht!]
	\input{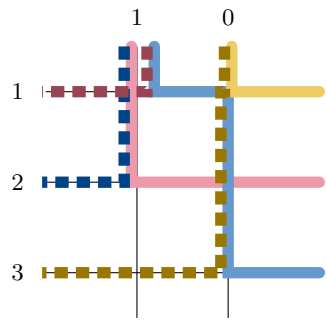}
	\caption{The system with boundary conditions given by $\lam =(1,1,0)$, $\wi{1} = (123)$, $\wi{2} = ()$, $\wi{3} = (12)$, and $\wi{4} = (123)$ has exactly one state. We compute $\wdd = (132)$, $\wc = (13)$. Note that this system has the same boundary permutations as the one in Figure~\ref{fig:fused_state_ex} which has many states, and only $\lam$ is modified.}
	\label{fig:fused_state_ex_mono}
\end{figure}

An example of such a monostatic system and its only state is shown in Figure~\ref{fig:monostatic_example}. We highlight two observations about this example. 
\begin{enumerate}
    \item In each row, the color on the right boundary, and dolor on the left boundary in that row enter the system as a pair on the top boundary. 
    \item Only three types of unfused vertices appear in the example: the fourth type in Figure~\ref{fig:monochrome_weights}, in which a color coming from the left and a dolor coming from the right meet and turn down as a pair, does not appear. 
\end{enumerate}
We will show that these two observations about this example are in fact generic characterizations of monostatic systems. The next result Proposition~\ref{prop:mono_conditions} makes use of the first observation. Recall the definition of $\wc,\wdd$ from Definition~\ref{def:wc}.

\begin{proposition}\label{prop:mono_conditions}
	For a system $\Si{}$, if $\wc=\wdd$ then the system is monostatic.
\end{proposition}

\begin{proof}[Proof of Proposition~\ref{prop:mono_conditions}]
	Let $j$ be the permutation $\wc = \wdd$. Then $(c_{\wi{1}^{-1}(i)}, d_{\wi{2}^{-1}(1)})$ the color-dolor pair exiting in the first row, agrees with the pair $(c_{\wi{3}^{-1}(j_1)}, d_{\wi{4}^{-1}\longelt(j_1)})$ which appears as a color-dolor pair on the top boundary. Now there is exactly one way to color the top row since this color-dolor pair must split at the top row and occupy all the horizontal edges. The same argument holds for each subsequent row, therefore the state is unique.
\end{proof}

A corollary shows that our second observation above also characterizes monostatic systems. 

\begin{corollary}\label{cor:fourth_type}
    If a system $\Si{}$ has a state in which no vertices of the fourth type in Figure~\ref{fig:monochrome_weights} appear, then the system is monostatic.
\end{corollary}

\begin{proof}
    Observe that since there are no vertices of the fourth type, there are exactly $r$ vertices of the third type, one in each row. By the argument of the proof of Proposition~\ref{prop:mono_conditions}, if the vertex of type 3 in the $k$th row is in column $\lam_{j}$, then $j = (j_1,j_2,\dots, j_r)$ defines a permutation such that $\wc=\wdd =j$.
\end{proof}

In fact, the conditions that $\wc = \wdd$ and that no vertex of the fourth type appears are equivalent, the forward implication based on the proof of Proposition~\ref{prop:mono_conditions} and the reverse direction based on the proof of Corollary~\ref{cor:fourth_type}

Finally we show that up to the choice of $\lam$, we have characterized all monostatic systems.

\begin{proposition}\label{lem:fourth_type}
	For a system $\Si{}$ with $\lam_i - \lam_{i+1} > N$ then the following conditions are equivalent:
	\begin{enumerate}
		\item The system is monostatic,
		\item $\wc= \wdd$,
		\item The system has a state in which no vertices of the fourth type in Figure~\ref{fig:monochrome_weights} appear.
	\end{enumerate}
	Dropping the condition on $\lam$, then statements $(2)$ and $(3)$ remain equivalent and imply $(1)$.
\end{proposition}

\begin{proof}[Proof of Proposition~\ref{lem:fourth_type}]
    We have shown that conditions (2) and (3) both imply condition (1); and that (2) and (3) equivalent by the proofs of Proposition~\ref{prop:mono_conditions} and Corollary~\ref{cor:fourth_type}. None of this depends on the choice of $\lam$.

    Now suppose that $\Si{}$ is a monostatic system with $\lam$ satisfying the condition $\lam - \lam_{i+1} > N$. Let $s$ be only state of the system $\Si{}$. Suppose there was a vertex of the fourth type appearing in $s$. If there is more than one such vertex, choose thee topmost, leftmost instance. Since $\lam$ is sufficiently dominant, we may move this vertex to a column to the left or right.
\end{proof}

\subsection{Systems with no states} In this section we classify boundary conditions for systems which have no states. Again, there are two ways for such a `failure' to occur.
\begin{enumerate}
    \item $\lam$ is not \emph{sufficiently dominant}.
    \item The orderings of the colors and dolors on the boundary edges make a state impossible to exist, regardless of $\lam$.
\end{enumerate}

For the first situation, again consider the systems in Figures~\ref{fig:fused_state_ex} and Figures~\ref{fig:fused_state_ex_mono}. These systems have fixed boundary conditions $\wi{1},\wi{2},\wi{3},\wi{4}$, and have $\lam = (2,2,0)$ and $\lam = (1,1,0)$, respectively. The former system has multiple states, and the latter has one state. If we consider a system with the same boundary permutations and $\lam = (0,0,0)$, this system has no states. 

The second situation is our main focus: the case that the permutations describing the order of colors and dolors on the boundary force a system to have no states. This has no dependence on $\lam$. Again, recall we use the boundary conditions defined in Definition~\ref{def:boundary_conditions} and notation $\wc,\wdd$ from Definition~\ref{def:wc}. The next result is most interesting for us.

\begin{proposition}\label{cor:states_dont_exist}
	If $\wdd > \wc$ or $\wdd$ and $\wc$ are incomparable in the Bruhat order, then the system has no states and its partition function is 0.
\end{proposition}

Proposition~\ref{cor:states_dont_exist} follows directly from the following Lemma~\ref{lem:states_exist_implies_bruhat}.

\begin{lemma}\label{lem:states_exist_implies_bruhat}
	If a system $\Si{}$ has at least one state, then $\wdd \leq \wc$ in the Bruhat order.
\end{lemma}

\begin{proof}[Proof of Lemma~\ref{lem:states_exist_implies_bruhat}]
	We proceed by induction on the number of vertices of the fourth type. A system with no such `bad' vertices has $\wdd=\wc$ by Lemma~\ref{lem:fourth_type}. We will show that removing a `bad' vertex as in Figure~\ref{fig:vertex_replace} has the effect of increasing the Bruhat order of $\wdd$ to this monostatic state.
	
	\begin{figure}[ht!]
		\input{figures/statesexistreplace}
		\caption{We remove a vertex of the fourth type by switching the dolors in row $i$ and $j$ on the left boundary. Call the left system $\Si{}$ and the right system $\Si{}'$. In particular, this swap is the same as replacing the permutation $\wi{2}= \wi{2}(\Si{})$ specifying the left boundary condition with $\wi{2}' = \wi{2}(\Si{}') = (i,j)\wi{2}$. This has the effect of replacing $\wdd(\Si{}) = \longelt\wi{4}\wi{2}^{-1}$ (left) with $\wdd' = \longelt\wi{4}(\wi{2}')^{-1}  = \wdd(i,j)$ (right). Multiplying by this transposition has the effect that $\wdd'$ has one more involution than $\wdd$ (since $i < j$ and $\wdd(i) < \wdd(j)$ but $\wdd(i) > \wdd(j)$), so $\wdd' > \wdd$.}
		\label{fig:vertex_replace}
	\end{figure}
	
	Consider a state $\Si{} = \S(\lam,\wi{1},\wi{2},\wi{3},\wi{4},\z)$. Starting with the topmost, leftmost instance of the fourth type of vertex in $\Si{}$, let $i$ be the row that the vertex appears in and $j$ be the row in which the involved dolor appears on the left boundary as in the left image of Figure~\ref{fig:vertex_replace}. To remove the `bad' vertex, we apply the transposition $t = (i,j)$ to $\wi{2}$ giving a new system $\Si{}' = \S(\lam,\wi{1},t\wi{2},\wi{3},\wi{4},\z)$ with all boundary conditions the same except the dolor ordering on the left side. Since $\wdd(\Si{}') = \longelt\wi{4}\wi{2}^{-1}t$, we have that $\wdd(\Si{}') = \wdd(\Si{})t$, or compressing notation, $\wdd' = \wdd t$. Finally, we argue that $\wdd < \wdd'$. We maintain that $i < j$ and $\wdd(i) < \wdd(j)$, as in the left of Figure~\ref{fig:vertex_replace}. On the right side, $\wdd'(i) > \wdd'(j)$ and all else is the same, so the inversion number of $\wdd'$ is one more than that of $\wdd$. Since $\ell(\wdd') > \ell(\wdd)$ and they differ by a transposition, $\wdd' > \wdd$.
\end{proof}

\begin{proposition}\label{prop:states_exist}
	A system $\Si{}$ has at least one state if and only if $\wdd \leq \wc$ in the Bruhat order and $\lam$ is \textit{sufficiently dominant.}
\end{proposition}

\begin{proof}[Proof of Proposition~\ref{prop:states_exist}]
	The forward direction is Lemma~\ref{lem:states_exist_implies_bruhat}. For the reverse direction we essentially reverse the proof of Lemma~\ref{lem:states_exist_implies_bruhat}. Since $\wdd \leq \wc$ in the Bruhat order, we can write $\wdd t_1 t_2 \cdots t_k = \wc$ where $t_i$ are transpositions and 
	\[
	\wdd < \wdd t_1  < \wdd t_1 t_2 < \cdots  < \wdd t_1 t_2 \cdots t_k  = \wc.
	\]
	Let $\Si{}^{(i)}$ be the system with boundary conditions $\lam,\wi{1},t_i\cdots t_1\wi{2},\wi{3},\wi{4}$. Since $\wdd t_1  \cdots t_k = \wc$, the state $\Si{}^{(k)}$ is monostatic. For the inductive step, suppose that $\Si{}^{(i)}$ has at least one state. We reverse the process in Figure~\ref{fig:vertex_replace} moving from a state of the right system $\Si{}^{(i)}$ to a state of the left system $\Si{}^{(i-1)}$ by applying the transposition $t_i$ to $\wi{2}(\Si{}^{(i)}) = t_i \cdots t_1 \wi{2}$. Note that to get a state of $\Si{}^{(i-1)}$ from a state of $\Si{}^{(i)}$ we have to make a choice of which column the newly added fourth type of vertex will appear in. The condition that $\lam$ is sufficiently dominant is required to ensure that there exists a column between $\lam_i$ (the column where $\wdd(i)$ appears) and $\lam_j$ (where $\wdd(j)$ appears) to insert the vertex.
\end{proof}

\section{The model is solvable}\label{sec:solvable}
In this section we show that the vertex model is \emph{solvable}. In our setting, we say that a model is solvable if it satisfies the \emph{Yang-Baxter equation}, illustrated in Figure~\ref{fig:ybe}. 

\begin{figure}[ht!]
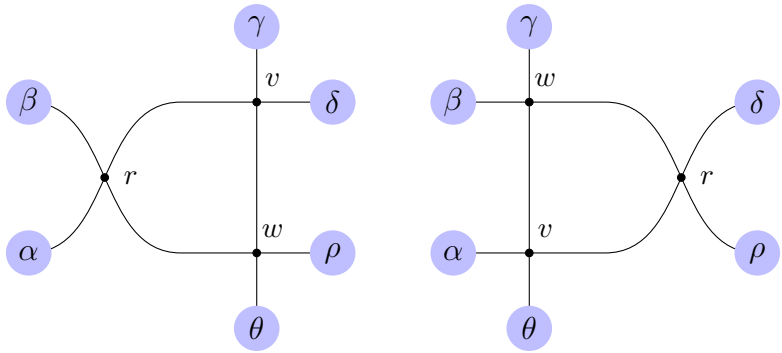

	\begin{center}
		\input{figures/ybelhs}
		\;\;\;\;\;
		\input{figures/yberhs}
	\end{center}
	\caption{The Yang-Baxter equation is satisfied if both of the above systems have the same partition function. This is for all choices of boundary conditions specified by horizontal spins $\alpha,\beta,\delta,\rho$, and vertical spins $\gamma,\theta$. The vertices $w,v$ are ordinary (fused) vertices, and $r$ is an $R$-vertex.}
	\label{fig:ybe}
\end{figure}

The Yang-Baxter equation is satisfied if for each choice of fixed boundary conditions, the two systems have the same partition functions. In this figure, $v$ and $w$ are ordinary fused vertices in our model, and $r$ is an $R$-vertex, which meets four \emph{horizontal} edges, rather than two horizontal edges and two vertical edges. Recall that the admissible ordinary vertices were constructed in Section~\ref{sec:fused_model}. In this section we will define admissible $R$-vertices, and give their Boltzmann weights. Then we will prove one of the major results of this paper, Theorem~\ref{thm:ybe}, that the model satisfies the Yang-Baxter equation.

\subsection{The Yang-Baxter equation is satisfied}
To show that the model is solvable, we will prove the following statement.

\begin{theorem}\label{thm:ybe_with_weights}
    The Yang-Baxter equation in Figure~\ref{fig:ybe} is satisfied by the Boltzmann weights for ordinary fused vertices provided in Section~\ref{sec:fused_model} and $R$-vertex weights in Figure~\ref{fig:rweights}.
\end{theorem}

\begin{figure}[ht!]
	\begin{center}
		 \input{figures/rweights}
	\end{center}
	\caption{$R$-vertex weights for the (fused) lattice model. Here, $a,b$ can be either colors or dolors, $z_i,z_j \in \C$.}
	\label{fig:rweights}
\end{figure}

\begin{corollary}\label{thm:ybe}
The model is solvable.
\end{corollary}

For the proof of Theorem~\ref{thm:ybe_with_weights}, we will actually return to the unfused model. Recall that a fused vertex is constructed by the ``fusion'' of many unfused vertices, as described in Section~\ref{sec:fused_model}. If we have a Yang-Baxter equation for the unfused model, we may use a tool called the \emph{train argument} to show that the fused Yang-Baxter equation is satisfied. The following two results will prove Theorem~\ref{thm:ybe_with_weights}.

\begin{lemma}\label{prop:mono_ybe}
	The unfused Yang-Baxter equation in Figure~\ref{fig:mono_ybe} is satisfied by the weights for unfused vertices in Figure~\ref{fig:monochrome_weights} and unfused $R$-vertices in Figure~\ref{fig:mono_rweights}.
\end{lemma}

\begin{figure}[ht!]
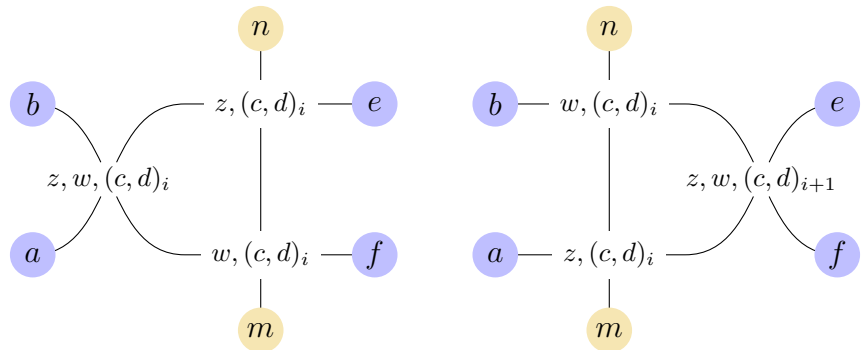

	\begin{center}
		\input{figures/monolhs}
		\;\;\;\;\;
		\input{figures/monorhs}
	\end{center}
	\caption{The Yang-Baxter equation for unfused vertices is satisfied if for all choices of boundary conditions $a,b,d,f \in \c \cup \d$, $n,m \in \N$, $(c,d)\in \c\times\d$, and $z,w\in \C$ the two systems above have the same partition function.}
	\label{fig:mono_ybe}
\end{figure}

\begin{lemma}\label{lem:ybe_train}
    If the unfused Yang-Baxter equation is satisfied, then the fused Yang-Baxter equation is satisfied with $R$-vertex given by the unfused $R$-vertex restricted to the pair $(c,d)_1$.
\end{lemma}

We prove Lemma~\ref{prop:mono_ybe} and Lemma~\ref{lem:ybe_train} in the following two sections, completing the proof of Theorem~\ref{thm:ybe_with_weights}.

\subsection{The unfused $R$-vertex and Yang-Baxter equation}
We define the unfused $R$-vertex, which connects four horizontal edges, depends on a color-dolor pair, and two spectral parameters in $\C$. The admissible unfused $R$-vertices and their weights are shown in Figure~\ref{fig:mono_rweights}

\begin{figure}[t!]
	\input{figures/rweightsmono}
	\caption{Unfused $R$-vertex weights.}
	\label{fig:mono_rweights}
\end{figure}

We now prove that the unfused weights satisfy the Yang-Baxter equation.

\begin{proof}[Proof of Lemma~\ref{prop:mono_ybe}]
	Consider the unfused Yang-Baxter equation in Figure~\ref{fig:mono_ybe}. We need at most two dolors and two colors for the pairs $(c,d)_{i}$ and $(c,d)_{i+1}$. The yellow vertical edges carry $n$ and $m$ bosons of type $(c,d)_i$. Since the capacity of horizontal edges is limited to one particle, $n$ and $m$ cannot differ by more than $2$. For the remaining boundary (and interior edges), there are at most two additional colors or dolors that can appear. All possibilities are covered if we consider these systems with 4 colors and 4 dolors. We checked the Yang-Baxter equation for $4$ colors and $4$ dolors, and $m \in [n-2,n+2]$ using a computer and so the Yang-Baxter equation in Figure~\ref{fig:mono_ybe} is proved.
\end{proof}

\subsection{The $R$-vertex satisfies the Yang-Baxter equation}
The $R$-vertex meets four horizontal edges, and has two spectral parameters. The admissible $R$-vertices for the (fused) model and their Boltzmann weights are given in Figure~\ref{fig:rweights}. 

Finally we show that these fused weights satisfy the Yang-Baxter equation, completing the proof that the model is solvable. Note that the admissible fused $R$-vertices and their weights agree exactly with the unfused $R$-vertices taking $(c,d) = (c,d)_1$ in the monochrome ordering of bosons given in Defintion~\ref{def:mono_ordering}. So, it is sufficient to show Lemma~\ref{lem:ybe_train}.

\begin{proof}[Proof of Lemma~\ref{lem:ybe_train}]
We will use the train argument to show that the fused Yang-Baxter equation is true. Consider the left hand image in Figure~\ref{fig:ybe}. Choose horizontal spins $\alpha,\beta,\delta,\rho \in \Sigma_h$, vertical spins $\gamma,\theta \in \Sigma_v$, and row parameters $z,w\in \C$. Recall that $\Sigma_v = \N^{rs}$, so we write $\gamma = (n_1,\dots,n_{rs})$ and $\theta = (m_1,\dots,m_{rs})$. We can visually expand the fused vertices $v,w$ into their constituent unfused vertices,
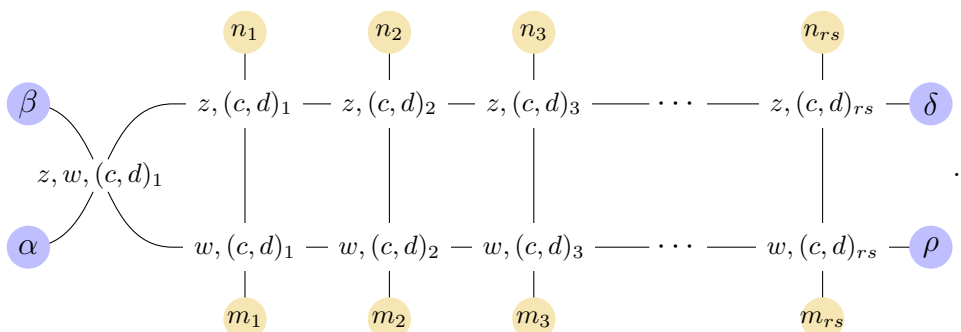
\[\input{figures/lhsybetrain}.\]
Now, since the fused $R$-vertex on the left agrees with the unfused $R$-vertex associated to the boson $(c,d)_1$, we can say that the following system has the same partition function by the unfused Yang-Baxter equation
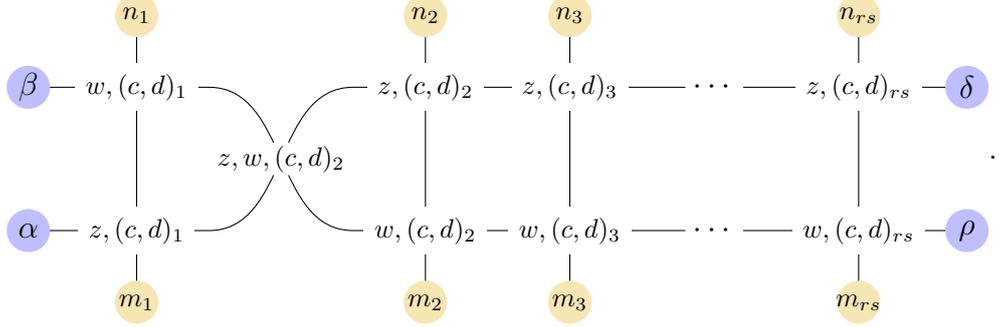
\[
\input{figures/lhsybetrain2}.
\]
We see that in fact we can repeatedly apply the unfused Yang-Baxter equation, eventually yielding the following system
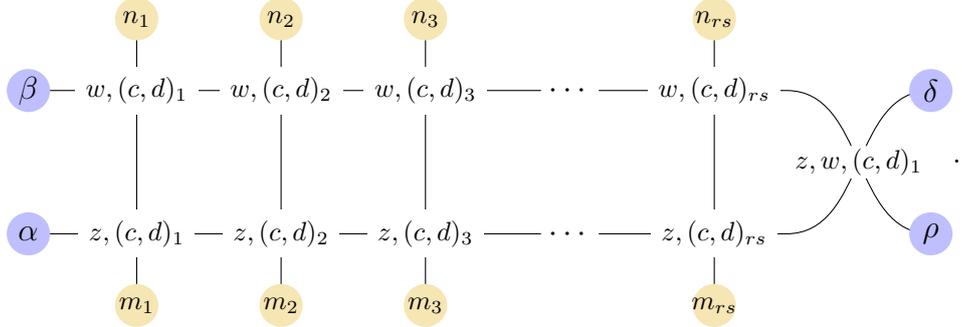
\[
\input{figures/rhsybetrain}.
\]
Applying the fusion process to the ordinary vertices and noting again that the fused $R$-vertex agrees with the unfused $R$-vertex for the boson $(c,d)_1$, we see that this system is the right hand image in Figure~\ref{fig:ybe}. We have shown that the left and right sides of the Yang-Baxter equation in  Figure~\ref{fig:ybe} have the same partition function, so we are done.
\end{proof}

\section{The partition function is determined by a recurrence relation}\label{sec:recurrence}

Recall that the partition function is a central object in the study of lattice models, as it often is the source of connections and applications to other fields. So far we have defined the partition function of a system $\Si{}$ as defined in Definition~\ref{def:boundary_conditions} as 
\[
Z(\Si{}) = \sum_{s\in\Si{}} \wt(s).
\]

We will show Theorem~\ref{thm:recurrence_determines_pt} that we can compute the partition function using a four-term recurrence relation. We will use the notation $Z(\wi{1},\wi{2};\z) = Z(\S(\lam, \wi{1},\wi{2},\wi{3},\wi{4}, \z)$ as shorthand for the partition function when all other boundary conditions are fixed.

\begin{theorem}\label{thm:recurrence_determines_pt}
    Let $\lam$ be a partition, $\wi{1},\wi{2},\wi{3},\wi{4}\in S_r$, and $\z\in \C^r$ be boundary conditions for a system of the lattice model, and $s_i \in S_r$ a simple reflection swapping $(i,i+1)$. Then the partition function satisfies the recurrence relation
    \begin{equation}\label{eq:recurrence}
        \begin{aligned}
        \begin{cases}
            (1-t)\,z_i\, Z(\wi{1},\wi{2};\z) \,+\, (z_i - z_{i+1} )\,Z(\wi{1},s_i\wi{2};\z) & \text{if } s_iw_2 < w_2 \\ 
            (1-t)\,z_{i+1}\, Z(\wi{1},\wi{2};\z) \,+\, t\,(z_i - z_{i+1} )\,Z(\wi{1},s_i\wi{2};\z) & \text{if } s_iw_2 > w_2  
        \end{cases} \;\;\;\;\;\;\;\;\; \\
        = \begin{cases}
            (1-t)\,z_i\, Z(\wi{1},\wi{2};s_i\z) \,+\, (z_i - z_{i+1} )\,Z(s_i\wi{1},\wi{2};s_i\z) & \text{if } s_iw_1 < w_1 \\ 
            (1-t)\,z_{i+1}\, Z(\wi{1},\wi{2};s_i\z) \,+\, t\,(z_i - z_{i+1} )\,Z(s_i\wi{1},\wi{2};s_i\z) & \text{if } s_iw_1 > w_1. 
        \end{cases}
        \end{aligned}
    \end{equation}
    
    Together with knowledge of monostatic systems and systems with no states, this recurrence completely determines the partition function of any system.
\end{theorem}

\newcommand{\di}[1]{D^{(#1)}_{i}}
    
\begin{corollary}
\label{cor:recursion_demazure}
    Define the following operators:
    \begin{align*}
        \di{1} &= (z_i- z_{i+1})^{-1}(z_i-z_is_i) \\
        \di{2} &= (z_i- z_{i+1})^{-1}(z_i-z_{i+1}s_i) \\
        \di{3} &= (z_i- z_{i+1})^{-1}(z_{i+1}-z_is_i) \\
        \di{4} &= (z_i- z_{i+1})^{-1}(z_{i+1}-z_{i+1}s_i).
    \end{align*}
    These satisfy the properties that 
    $(\di{1})^2 = \di{1}$, 
    $(\di{2})^2 = \di{2}$, 
    $(\di{3})^2 = -\di{3}$,  
    $(\di{4})^2 = -\di{4}$, and
    \[
    \di{1}\di{2} = \di{1}\di{3} = \di{2}\di{4} = \di{3}\di{1} = \di{4}\di{2} = \di{4}\di{3} = 0.
    \]
    Then the partition function satisfies
    \begin{align*}
        \di{1}Z(\wi{1},\wi{2};\z) &= (1-t)^{-1} (Z(s_i\wi{1},\wi{2};s_i\z) - Z(\wi{1},s_i\wi{2};\z)) & \textrm{if } s_i\wi{1} < \wi{1}, s_i\wi{2} < \wi{2}\\
        \di{2}Z(\wi{1},\wi{2};\z) &= (1-t)^{-1} (tZ(s_i\wi{1},\wi{2};s_i\z) - Z(\wi{1},s_i\wi{2};\z)) & \textrm{if } s_i\wi{1} > \wi{1}, s_i\wi{2} < \wi{2}\\
        \di{3}Z(\wi{1},\wi{2};\z) &= (1-t)^{-1} (Z(s_i\wi{1},\wi{2};s_i\z) - t Z(\wi{1},s_i\wi{2};\z)) & \textrm{if } s_i\wi{1} < \wi{1}, s_i\wi{2} > \wi{2}\\
        \di{4}Z(\wi{1},\wi{2};\z) &= (1-t)^{-1}(tZ(s_i\wi{1},\wi{2};s_i\z) - tZ(\wi{1},s_i\wi{2};\z)) & \textrm{if } s_i\wi{1} > \wi{1}, s_i\wi{2} > \wi{2}.\\
    \end{align*}
\end{corollary}

The goal of this section is to prove Theorem~\ref{thm:recurrence_determines_pt}. First we will show that the partition function satisfies a four-term recurrence relation. Then, we will show that knowledge of the partition functions of the monostatic systems and systems with no states from Section~\ref{sec:monostatic_systems} serve as base cases for the relation. Finally we will show that these data fully determine the partition function of any system of the model.

Then Corollary~\ref{cor:recursion_demazure} can be easily verified using the recursion \eqref{eq:recurrence} and noticing that $Z(\wi{1},\wi{2};\z) = s_iZ(\wi{1},\wi{2};s_i\z)$. This result expresses the recursion as divided difference operators on the partition function. As we will see in Section~\ref{sec:llp}, the bicolored models generalize certain colored bosonic models whose partition functions are shown to satisfy Demazure relations (Proposition~4.5 of \cite{BN}), drawing connections with $p$-adic representation theory. When there is exactly one dolor, the relations on the bicolored models agree with those on the colored models. These divided difference operators on the partition functions of the bicolored models hint at a more general connection with the representation theory of $p$-adic
groups.

\subsection{The recurrence relation} 
First we show that the partition function satisfies the recurrence relation \eqref{eq:recurrence}.

\begin{proposition}\label{prop:pt_satisfies_recurrence}
    The recurrence relation \eqref{eq:recurrence} in Theorem~\ref{thm:recurrence_determines_pt} is satisfied.
\end{proposition}

Let $\lam, \wi{1},\wi{2},\wi{3},\wi{4}, \z$ as in Definition~\ref{def:boundary_conditions}. Consider the grid in Figure~\ref{fig:recurrence_grid}; this is the standard grid for the vertex model, with one modification being the addition of an $R$-vertex on the left boundary at rows $1$ and $2$. We may label the boundary edges of this grid according to the rules specified in Definition~\ref{def:boundary_conditions}: $\wi{1}$ specifies the right boundary edges, $\wi{2}$ specifies the left boundary edges, $\lam,\wi{3},\wi{4}$ specify the top boundary, and the row parameter $\z$ labels the rows as before. We call this system $\S_1^L(\lam, \wi{1},\wi{2},\wi{3},\wi{4}, \z)$. 

\begin{figure}[ht!]
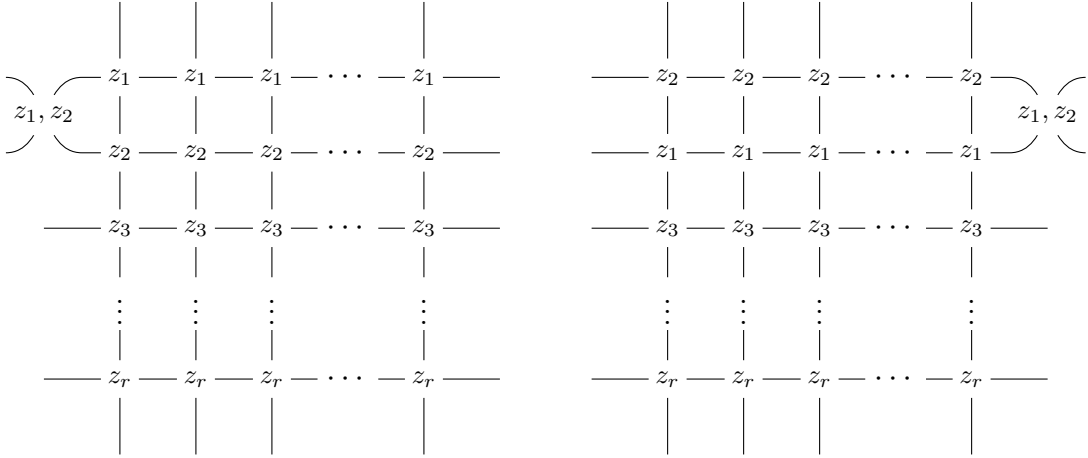

    \centering
    \input{figures/recurrencetrain1}\;\;\;\;\;\;\;\;
    \input{figures/reecurrencetrain2}
    \caption{The grids for the systems $\S_1^L$ (left) and $S_1^{R}$ (right).}
    \label{fig:recurrence_grid}
\end{figure}

In general we define the system $\S_i^{L}(\lam, \wi{1},\wi{2},\wi{3},\wi{4}, \z)$ by taking a grid of appropriate dimensions $\numrows \times \numcols$, adding an $R$-vertex connecting the left two boundary edges on rows $i,i+1$, and then labelling the boundary edges of the modified grid by $\lam, \wi{1},\wi{2},\wi{3},\wi{4}, \z$ as in Definition~\ref{def:boundary_conditions}. Similarly, we define the system $\S_i^{R}(\lam, \wi{1},\wi{2},\wi{3},\wi{4}, \z)$ by starting with a grid of size $\numrows \times \numcols$, attach an $R$-vertex connecting the right boundary edges in rows $i,i+1$, and label the boundaries on this modified grid by $\lam, \wi{1},\wi{2},\wi{3},\wi{4},s_i\z$ as in Definition~\ref{def:boundary_conditions}. See Figure~\ref{fig:recurrence_grid} for the example of the grids for $\S_1^L$ and $\S_1^R$. 

Note that in the $\S_i^R$ compared to $\S_i^L$, the rows labelled by $z_i,z_{i+1}$ are apparently swapped. We will see that this choice allows us to relate their partition functions. 
\begin{lemma}\label{lem:sir}
    For fixed boundary data $(\lam, \wi{1},\wi{2},\wi{3},\wi{4}, \z)$ and $i\in \{1,\cdots, r-1\}$,
    \[
        Z(\S_i^L(\lam, \wi{1},\wi{2},\wi{3},\wi{4}, \z)) = Z(S_i^R(\lam, \wi{1},\wi{2},\wi{3},\wi{4}, \z)).
    \]
\end{lemma}

\begin{proof}[Proof of Lemma~\ref{lem:sir}]
    This result follows from the train argument, similar to our proof of the fused Yang-Baxter equation in Lemma~\ref{lem:ybe_train}. Starting with $\S_i^L(\lam, \wi{1},\wi{2},\wi{3},\wi{4}, \z)$, we repeatedly apply the Yang-Baxter equation to ``slide'' the $R$-vertex across each column. This process yields a sequence of systems with an $R$-vertex each in a different location, and swapping the parameters $z_i,z_{i+1}$ in the wake of the $R$-vertex to its left. Each of the systems in this sequence have the same partition function. After applying the Yang-Baxter equation $N$ times, we are left with the system $S_i^R(\lam, \wi{1},\wi{2},\wi{3},\wi{4}, \z)$. 
\end{proof}

Now we will prove Proposition~\ref{prop:pt_satisfies_recurrence}.

\begin{proof}[Proof of Proposition~\ref{prop:pt_satisfies_recurrence}.]
    Let $\lam, \wi{1},\wi{2},\wi{3},\wi{4}, \z$ be the boundary conditions for the system $\S$. Consider the system $\S_i^L(\lam, \wi{1},\wi{2},\wi{3},\wi{4}, \z)$. Its $R$-vertex has two edges specified by the boundary condition, as in Figure~\ref{fig:train_recurrence}. 
    
    \begin{figure}[ht!]
        \begin{center}
            \input{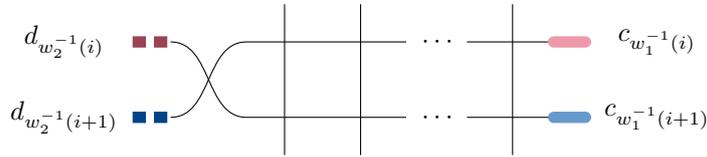}
        \end{center}
        \caption{Rows $i$ and $i+1$ of $\S_i^L$.}
        \label{fig:train_recurrence}
    \end{figure}
    There are two admissible choices for the remaining edges: either the dolors cross, or they do not. We can use this restriction on the $R$-vertex to decompose the partition function of $\S_i^L$ into two cases. We have that
    \begin{equation}\label{eq:recurrence_lhs}
    Z(\S_i^L) = \wt\left({\rvertexdolsametwo}\right) \, Z(\wi{1},\wi{2};\z)\, +\, \wt\left({\rvertexdolcrosstwo}\right) \, Z(\wi{1},s_i\wi{2};\z).
    \end{equation}
    
    Now consider $\S_i^R(\lam, \wi{1},\wi{2},\wi{3},\wi{4}, \z)$. We can similarly decompose its partition function based on the two admissible configurations of its $R$-vertex,
    \begin{equation}\label{eq:recurrence_rhs}
    Z(\S_i^R) = Z(\wi{1},\wi{2};s_i\z)\, \wt\left(\rvertexcolsametwo\right) \,+\,  Z(s_i\wi{1},\wi{2};s_i\z)\, \wt\left( \rvertexcolcrosstwo\right).
    \end{equation}
    By Lemma~\ref{lem:sir}, we have $Z(\S_i^L)= Z(\S_i^R)$. Thus we can say that the left hand sides of \eqref{eq:recurrence_lhs} and \eqref{eq:recurrence_rhs} are equal, giving a four-term recurrence relation for the partition function.
    Substituting the weights of the $R$-vertices into the recurrence, we are done.
\end{proof}

\subsection{Partition functions of monostatic systems}
To use the recurrence relation we proved in the previous subsection to compute partition functions, we will use knowledge of the partition function of systems with no states and systems with exactly one state as the base cases. Systems with no states of course have partition function equal to zero. We will compute the partition function of monostatic systems, which are systems with exactly one state. 

\begin{proposition}\label{prop:mono_pt_function}
	For a monostatic system $\Si{}$ satisfying the conditions of Proposition~\ref{lem:fourth_type}, its partition function is
	\[
	Z(\Si{}) =t^{s_d \,+\,s_c} \,\cdot\, \prod_{i=1}^r z_i^{N(r-1)+\lam_{j_i}},
	\]
	where $j_i = \wc(i)$, 
	\[
	s_c = \#\{k > j_i \,:\, \wi{3}(k) > \wi{1}(i), \wc^{-1}(k) > i\},
	\]
	and
	\[
	s_d =\#\{k < j_i \,:\, \wi{4}^{-1}\longelt(k) < \wi{2}^{-1}(i), \wdd^{-1}(k) > i\}.
	\]
\end{proposition}

\begin{proof}
	Consider row $i$ of a monostatic state of the system. We consider the instances of each type of (monochrome) vertex in this row. By Proposition~\ref{lem:fourth_type}, no instances of the fourth type of vertex in the state at all, let alone in any given row. There is exactly one instance of the third type of monochrome vertex in the row, where the color $c_{\wi{1}^{-1}(i)}$ and dolor $d_{\wi{2}^{-1}(i)}$ which exit at the right and left of this row split off. This vertex appears in column $\lam_{j_i})$ where $j_i = \wc(i)$ (or $\wdd(i)$), as $\wc$ (and $\wdd$) locate the positions of $c_{\wi{1}^{-1}(i)}$ (and $d_{\wi{2}^{-1}(i)}$) on the top edge, and $\wc=\wdd$. This instance in row $i$ of the third type of monochrome vertex in Figure~\ref{fig:monochrome_weights} has weight 1 since our boundary conditions allow for only one instance of each color and dolor. The remaining (monochrome) vertices are of the first two types.
	
	First, we count factors of $z_i$ in the partition function. To the right of the fused vertex in column $\lam_{j_i}$, there is a factor of $z_i$ exactly when the color $c_{\wi{1}^{-1}(i)}$ matches the color of a monochrome vertex. There are $\lam_{j_i}$ fused vertices to the right of $\lam_{j_i}$ (labelled $\lam_{j_i}-1$ to $0$), and each contains $r$ monochrome vertices labelled with the color $c_{\wi{1}(i)}$ (for each pair with the $r$ dolors); this contributes a factor of $z_i^{r \lam_{j_i}}$. Similarly, when the dolor $d_{\wi{2}^{-1}(i)}$ crosses a fused vertex as in the first type in Figure~\ref{fig:monochrome_weights}, there is one factor of $z_i$ for each dolor different from $d_{\wi{2}(i)}$. To the left of the vertex in column $\lam_{j_i}$, there are $N-\lam_{j_i}-1$ fused vertices (labelled $N-1$ to $\lam_{j_i}+1$), and there are $r-1$ pairs of a dolor different from $d_{\wi{2}(i)}$ with the color $c_1$; this contributes a factor of $z_i^{(r-1)(N-\lam_{j_i}-1)}$. 
	
	Finally, consider the fused vertex in column $\lam_{j_i}$. Recall that each fused vertex is composed of blocks of monochrome vertices labelled by the same dolor. In each of the dolor blocks to the left of the block labelled $d_{\wi{2}^{-1}(i)}$, these are monochrome vertices of the first type in Figure~\ref{fig:monochrome_weights}, which as before contribute exactly one factor of $z_i$ for the pair with the color $c_1$. In each of the dolor blocks to the right of $d_{\wi{2}^{-1}(i)}$, these are the second type of monochrome vertex, and we again have one factor of $z_i$ per block when the colors agree with $c_{\wi{1}^{-1}(i)}$. In the block labelled by $d_{\wi{2}^{-1}(i)}$, note that it is impossible for a vertex of either type to contribute a factor of $z_i$: in the first type we are in the case where the dolor label of the vertex matches the dolor on the horizontal edge which always has weight 1, and for the second type the only column which contributes a factor $z_i$ is the one labeled $(c_{\wi{1}^{-1}(i)},d_{\wi{2}^{-1}(i)})$ which we have already seen is be the third type of vertex in Figure~\ref{fig:monochrome_weights}. Thus this column contributes another $z_i^{r-1}$ to the partition function. In total we have $z_i^{r\lam_{j_i}} \cdot z_i^{(N-\lam_{j_i}-1)(r-1)} \cdot z_i^{r-1} = z_i^{N(r-1)+\lam_{j_i}}$. 
	
	Finally, we count the factors of $t$. These come from crossings of occupied vertical edges with horizontal edges. Note that since we have exactly one instance of each color and dolor, each time there is a contributing crossing, this is one additional factor of $t$ in the partition function. On the dolor side, a contribution comes from a crossing of $d_{\wi{2}^{-1}(i)}$ with a pair $(c,d)$ with $d < d_{\wi{2}^{-1}(i)}$. The dolors to the left of $d_{\wi{2}^{-1}(i)} = d_{\wi{4}^{-1}w_0(j_i)}$ on the top edge are $d_{\wi{4}^{-1}w_0(k)}$ for indices $k < j_i$. There is a crossing when the dolor $d_{\wi{4}^{-1}w_0(k)}$ appears below $d_{\wi{2}^{-1}(i)}$ on the left boundary edge: this is the condition that $\wi{2}\wi{4}^{-1}w_0(k)= \wdd^{-1}(k) > i$. The sum $s_d$ counts exactly how many of these crossings occur.
	Similarly, on the colored horizontal edges, there is a contribution of $t$ to the partition function when there is a crossing with a color greater than $c_{\wi{1}^{-1}(i)}$. The colors to the right of $c_{\wi{1}^{-1}(i)}$ on the top boundary are $c_{\wi{3}^{-1}k)}$ for $k > j_i$, and there is a crossing when $c_{\wi{3}^{-1}(k)}$ appears below $c_{\wi{1}(i)}$ on the right edge. So we need both that $c_{\wi{3}^{-1}(k)} > c_{\wi{1}^{-1}(i)}$ and $\wi{1}\wi{3}^{-1}(k) = \wc^{-1}(k) > i$; these are counted by $s_c$.
\end{proof}

\subsection{The recurrence relation determines the partition function}
Finally we prove Theorem~\ref{thm:recurrence_determines_pt}. We have shown in Proposition~\ref{prop:pt_satisfies_recurrence} that the partition function satisfies the recurrence relation \eqref{eq:recurrence}. We first prove the following lemma.

\begin{lemma}\label{lem:recurrence_bruhat}
	Let $S \subset W\times W$ be the set of all pairs $(u,v)$ such that $u \not< v$ under the Bruhat order. Let $s$ be a simple reflection in $W$. Suppose that if $(us,vs), (u,vs)$ are in $S$ then $(u,v)$ is in $S$. Then $S = W\times W$. 
\end{lemma}

\begin{proof}[Proof of Lemma~\ref{lem:recurrence_bruhat}]
	We induct on the length of $v$. First if $v = id$, then for any $u\in W$ the pair $(u,v)$ is in $S$ (because $u > v$). Now consider the pair $(u,v)$, and suppose that for $v' < v$ all pairs $(u',v')$ are in $S$. Let $s$ be a right descent of $v$, so that $vs < v$. Then $(us,vs)$ and $(u,vs)$ are in $S$, so $(u,v)$ is in $S$ too.
\end{proof}

\begin{proof}[Proof of Theorem~\ref{thm:recurrence_determines_pt}]
    Let $S$ be the subset of $\W\times \W$ consisting of pairs $(u,v)$ corresponding to $(\wdd(\Si{}), \wc(\Si{}))$ of systems $\Si{}$ whose partition functions are known. We will show that $S= \W\times \W$, which will prove the result.
    
    We have already exactly computed the partition functions of systems with no states and systems with one state in Corollary~\ref{cor:states_dont_exist} and Propositions~\ref{prop:mono_pt_function}, respectively. By Theorem~\ref{prop:mono_conditions}, the quantity of states of a system $\Si{}$, and therefore our ability to compute its partition function, is determined only by $\wc(\Si{})$ and $\wdd(\Si{})$. Therefore $S$ contains all pairs $u \not< v$.
    
    From the four terms of the recurrence relation in Theorem~\ref{thm:recurrence_determines_pt}, we can add additional relations to $S$. We can compute the partition function for any one of the four systems the recurrence if the remaining three systems have known partition functions. Consider the pairs $(\wdd(\Si{}), \wc(\Si{}))$ for each system $\Si{}$ in \eqref{eq:recurrence}: $(\wdd, \wc)$, $(\wdd s_i, \wc)$,  $(\wdd, \wc)$,  $(\wdd, \wc s_i)$, respectively. Therefore we may add the relations that if the first two pairs or second two pairs of the trio $\{(us, v),(u, v),(u, vs)\}$ are in $S$ then the third pair is in $S$. Sending $v \mapsto vs$, this is a condition for Lemma~\ref{lem:recurrence_bruhat}, completing the proof.
\end{proof}

\section{Relation to colored and uncolored bosonic models}\label{sec:llp}
In this section we show a \emph{color merging} process for the bicolored bosonic lattice models, which gives relations between bicolored models, colored models, and uncolored models. When there is exactly one dolor, the partition function of the model matches that of Bump and Naprienko's colored bosonic models \cite{BN} by considering the one dolor to be $+$ in their notation. When is more than one dolor, the color merging process we describe in the following section will give a procedure of ``merging'' the dolors to recover Bump and Naprienko's colored model from the bicolored model we defined in this paper.

We prove this \textit{color merging} process locally, showing that the weights of individual bicolored fused vertices lift to the colored weights. Remarkably, this local process has global implications. We show that our bicolored bosonic model's weights agree with the weights for a colored bosonic model when we forget the dolors; and that the weights agree with an uncolored bosonic model when we forget both the colors and dolors. Concretely, in Corollary~\ref{cor:llp_col_pt_fn}, we see how the partition function of a colored model is the sum of partition functions of bicolored models, where the sum is over possible boundary permutations of the dolors. In Corollary~\ref{cor:llp}, we sum over the boundary permutations of the colors to recover the uncolored models. These global results draw connections to known applications of the partition functions of colored and uncolored models to areas including $p$-adic representation theory and symmetric functions.

\subsection{Relation between bicolored and colored models}

We give a procedure to map the bicolored models to colored models. We define the maps on the spinsets of horizontal and vertical edges of the models $p_h: \Sigma_h^\text{bicol} \to \Sigma_h^\text{col}$ and $p_v: \Sigma_v^\text{bicol} \to \Sigma_v^\text{col}$ respectively. The map on horizontal spins is defined by 
\[
p_h(d_i) = +,\;\;\;\;\;\; p_h(c_i) = c_i.
\]
On vertical spins, $p_v$ sends a bicolored vertical edge with multiplicity of each color-dolor pair $(c,d)$ given by $n(c,d)$ to the colored edge 
\[
c_1^{m_1}c_2^{m_2}\cdots c_r^{m_r}, \;\;\;\; m_i = \sum_{d \in {\bf d}} n(c_i,d).
\]

The result is the following. The notation $\beta^\text{bicol}$ gives the Boltzmann weight of a bicolored fused vertex with monochrome weights from Figure~\ref{fig:monochrome_weights}, and $\beta^\text{col}$ gives the Boltzmann weight of the colored fused vertex with monochrome weights given by the $R$-weights in Bump and Naprienko \cite{BN}.

\begin{proposition}\label{prop:llp}
	Let $(A,B,C,D) \in \Sigma_h^\text{col} \times \Sigma_v^\text{col} \times \Sigma_h^\text{col} \times \Sigma_v^\text{col}$, and let $(b,c) \in \Sigma_v^\text{bicol} \times \Sigma_h^\text{bicol}$ such that $p_v(b) = B$ and $p_h(c) = C$. Then
	\[
		\beta^\text{col} \left(\gambig \right) = z^{1-r} \sum_{\substack{(a,d) \in \Sigma_h^\text{bicol} \times \Sigma_v^\text{bicol} \\ p_h(a) = A \\ p_v(d) =D}} \beta^\text{bicol}\left( \gamlittle \right).
	\]
\end{proposition}

\begin{proof}
	Consider a model with $s$  colors and $r$ dolors.
	Instead of the maps $p_h,p_v$, we will consider the maps $p_h': \Sigma_h^\text{bicol} \to \Sigma_h^\text{bicol}$ and $p_v: \Sigma_v^\text{bicol} \to \Sigma_v^\text{bicol}$ which instead of sending colors to colors and dolors to $+$, will send all dolors to $d_r$ the maximal dolor. The map on horizontal spins is defined by 
	\[
	p_h(d_i) = d_r,\;\;\;\;\;\; p_h(c_i) = c_i.
	\]
	On vertical spins, $p_v'$ sends a bicolored vertical edge to the edge whose bosons of type $(c,d_r)$ have multiplicity $\sum_{d \in {\bf d}} n(c,d)$, and all other edges $(c,d_i)$ for $d_i \neq d_r$ having multiplicity 0. Specifically, sending an edge with bosons with multiplicities $n(c,d) > 0$ to the edge
	
	\[
	(c_1d_r)^{m_1}(c_2d_r)^{m_2}\cdots (c_sd_r)^{m_s}(c_1d_{r-1})^{0}\cdots (c_sd_1)^{0}, \;\;\;\; m_i = \sum_{d \in {\bf d}} n(c_i,d).
	\]
	
		\begin{figure}
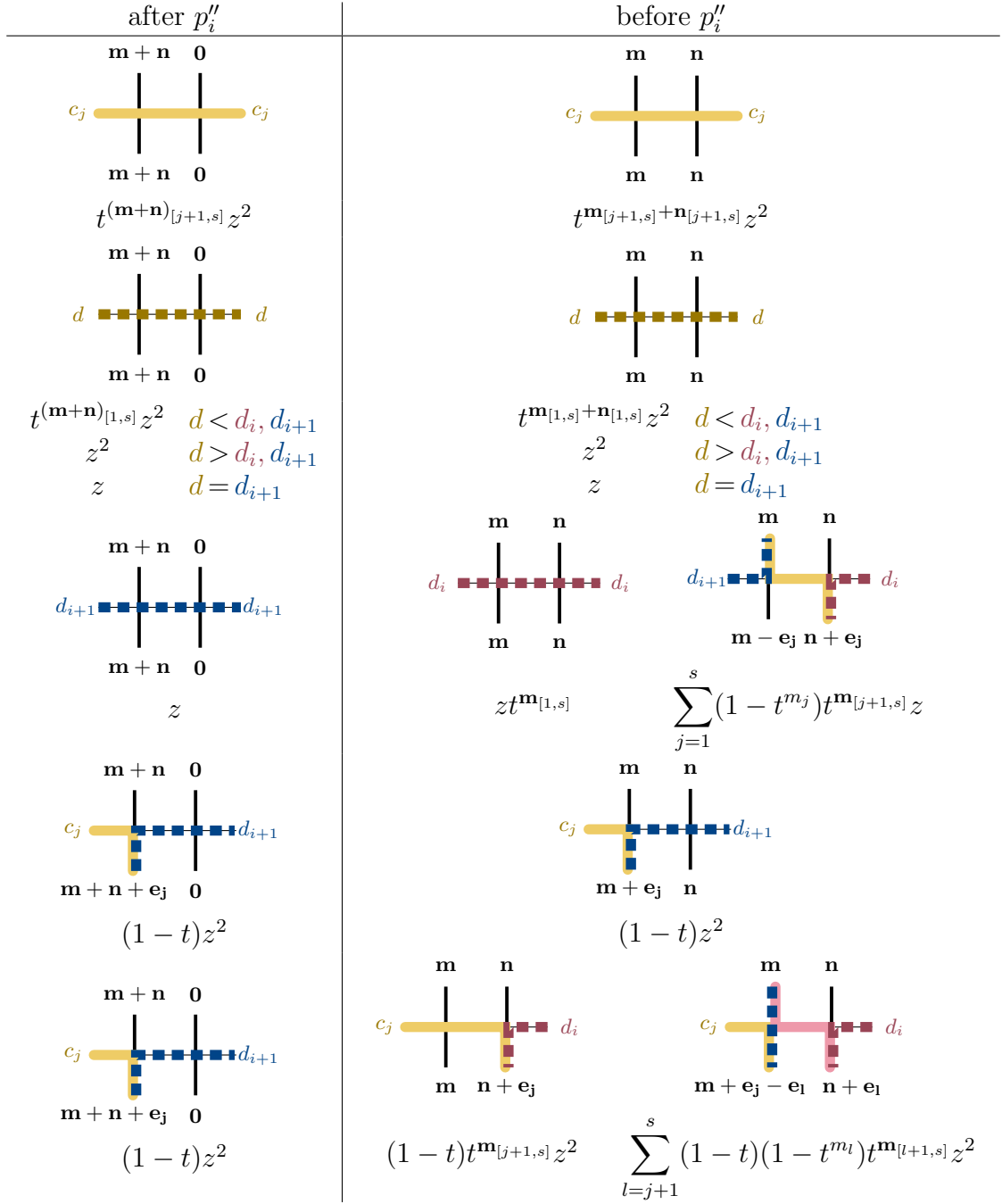

		\begin{center}
			\begin{tabular}{c|c}
				after $p_i''$
				&
				before $p_i''$ \\ \hline
				
				\input{figures/llp_lifted_cc}
				&
				\input{figures/llp_cc} \\
				$ t^{{\bf (m + n)}_{[j+1,s]}} z^2$ & $t^{{\bf m}_{[j+1,s]}+ {\bf n}_{[j+1,s]}} z^2$ \\
				
				\input{figures/llp_lifted_d_d}
				&
				\input{figures/llp_d_d} \\
				
				\begin{tabular}{cl}
					$t^{{\bf (m + n)}_{[1,s]}}z^2$ & $\color{cb-dark-yellow}{d} \, \color{black}{<}\, \color{cb-dark-red}{d_i}, \color{cb-dark-blue}{d_{i+1}}$ \\
					$z^2$ & $\color{cb-dark-yellow}{d}\, \color{black}{>}\, \color{cb-dark-red}{d_i}, \color{cb-dark-blue}{d_{i+1}}$ \\
					$z$ & $\color{cb-dark-yellow}{d}\, \color{black}{=}\, \color{cb-dark-blue}{d_{i+1}}$
				\end{tabular}
				&
				\begin{tabular}{cl}
					$t^{{\bf m}_{[1,s]}+ {\bf n}_{[1,s]} } z^2$ & $\color{cb-dark-yellow}{d} \, \color{black}{<}\, \color{cb-dark-red}{d_i}, \color{cb-dark-blue}{d_{i+1}}$ \\
					$z^2$ & $\color{cb-dark-yellow}{d}\, \color{black}{>}\, \color{cb-dark-red}{d_i}, \color{cb-dark-blue}{d_{i+1}}$ \\
					$z$ & $\color{cb-dark-yellow}{d}\, \color{black}{=}\, \color{cb-dark-blue}{d_{i+1}}$
				\end{tabular} \\
				
				\begin{tabular}{c}
					\input{figures/llp_lifted_dd} \\
					$z$
				\end{tabular}
				&
				\begin{tabular}{cc}
					\input{figures/llp_dd_1} & \input{figures/llp_dd_2} \\
					$z t^{{\bf m}_{[1,s]}}$ & $\displaystyle \sum_{j=1}^s (1-t^{m_j})t^{{\bf m}_{[j+1,s]}} z$
				\end{tabular} \\
				\begin{tabular}{c}
					\input{figures/llp_lifted_cd} \\
					$(1-t)z^2$
				\end{tabular}
				&
				\begin{tabular}{c}
					\input{figures/llp_cd_p1} \\
					$(1-t)z^2$
				\end{tabular} \\
				
				\begin{tabular}{c}
					\input{figures/llp_lifted_cd} \\
					$(1-t)z^2$ \\
				\end{tabular}
				&
				\begin{tabular}{cc}
					\input{figures/llp_cd_1} & \input{figures/llp_cd_2}\\
					$(1-t)t^{{\bf m}_{[j+1, s]}}z^2$ & 
					$\displaystyle \sum_{l = j+1}^s (1-t)(1-t^{m_l})t^{{\bf m}_{[l+1,s]}}z^2$
				\end{tabular} \\
			\end{tabular}
		\end{center}
		\caption{(Part I of II) We check casewise that eliminating the dolor $d_i$ by the map $p''$ does not change the weight of a fused vertex. The two columns represent blocks of monochrome vertices of colors $c_1,\cdots,c_s$ paired with the dolors $d_{i+1}$ (left) and $d_i$ (right). The vector ${\bf m}$ represents the number of bosons of type ${\bf c}d_{i+1}$, and similarly ${\bf n}$ for bosons of type ${\bf c}d_{i}$. We use the notation ${\bf m}_{[a,b]}$ as shorthand for the sum $\sum_{l = a}^b m_l$. In the cases when one image represents several of states (for example the bottom right corner), the weight listed below is the sum of the weights of all of the states.}
		\label{fig:llp1}
	\end{figure}
	
	\begin{figure}
		\begin{center}
			\begin{tabular}{c|c}
				after $p_i''$
				&
				before $p_i''$ \\ \hline
				
				\begin{tabular}{c}
					\input{figures/llp_lifted_dc} \\
					\small $\displaystyle \frac{(1-t^{m_j + n_j})}{1-t}t^{{\bf m+n}_{[j+1,s]}}z$
				\end{tabular}
				&
				\begin{tabular}{cc}
					\input{figures/llp_dc_1} & 	\input{figures/llp_dc_2} \\
					\small $\displaystyle \frac{(1-t^{m_j})}{1-t}t^{{\bf m+n}_{[j+1,s]}}z$ & 
					\small $\displaystyle \frac{(1-t^{n_j})}{1-t}t^{{\bf m}_{[1,s]}+{\bf n}_{[j+1,s]}}z$ 
				\end{tabular} \\
                &
                \begin{tabular}{c}
                    \input{figures/llp_dc_3} \\
                    \small $\displaystyle \sum_{l=1}^{j-1} \frac{(1-t^{m_l})(1-t^{n_j})}{1-t} t^{{\bf m + n}_{[j+1,s]}} z $
                \end{tabular}

			\end{tabular}
		\end{center}
		\caption{(Part II of II) We check casewise that eliminating the dolor $d_i$ by the map $p''$ does not change the weight of a fused vertex.  All definitions and notation is the same as describled for Figure~\ref{fig:llp1}.}
		\label{fig:llp2}
	\end{figure}
	
	Observe that the maps $p'$ and $p$ are related. In particular since the weights of the colored and bicolored model agree when there is only one dolor in the bicolored model, the only difference is that these extra monochrome edges carrying 0 bosons remaining to the right of $d_r$ under the map $p'$. The difference between the results of these maps is accounted for by the weight of these `extra' monochrome vertices that appear in the output of $p'$ but not $p$. Since all the `extra' monochrome vertices carry zero bosons, the contributions to the weight are exactly one factor $z$ per dolored block. So the weights of the images of a vertex under these maps are related by
	\[
	\beta^\text{col}\left(p\left(\gamlittle\right)\right) = \frac{1}{z^{r-1}}\,\beta^\text{bicol}\left(p'\left(\gamlittle\right)\right).
	\]
	
	Now we will show that the map $p'$ preserves the weight of a fused vertex by defining maps $p_i''$ which essentially send all of the bosons involving $d_i$ to $d_{i+1}$. Specifically, on horizontal spins
	\[
	p_{i,h}'': d_i \mapsto d_{i+1}
	\]
	and preserves all other spins,
	and on vertical spins
	\[
	p_{i,v}'': \begin{cases}
		n(c,d_{i+1}) \mapsto n(c,d_{i}) + n(c,d_{i+1}) \\
		n(c,d_{i}) \mapsto 0,
	\end{cases}
	\]
	and preserves all other multiplicities. 
	
	The point of this is to factor $p' = p_1'' \circ p_2'' \circ \cdots \circ p_{r-1}''$ as the composition of $r-1$ maps by successively eliminating dolors. We will show that each of these $p_i''$ maps preserve the weight of a dolored vertex, therefore the composition $p'$ preserves the weight. In particular, we will show the following color merging property for the maps $p_i''$. 
	Let $(A,B,C,D) \in \Sigma_h^\text{bicol} \times \Sigma_v^\text{bicol} \times \Sigma_h^\text{bicol} \times \Sigma_v^\text{bicol}$, and let $(b,c) \in \Sigma_v^\text{bicol} \times \Sigma_h^\text{bicol}$ such that $p_{i,v}''(b) = B$ and $p_{i,h}''(c) = C$. Then
	\[
	\beta^\text{bicol} \left(\gambig \right) = \sum_{\substack{(a,d) \in \Sigma_h^\text{bicol} \times \Sigma_v^\text{bicol} \\ p_{i,h}''(a) = A \\ p_{i,v}''(d) =D}} \beta^\text{bicol}\left( \gamlittle \right).
	\]
	
	We check this color merging property casewise in Figures~\ref{fig:llp1} and \ref{fig:llp2}. We will check that in the vertical columns corresponding to bosons including $d_i, d_{i+1}$ the weight is preserved before and after applying the map $p_i''$. We consider only these columns since the columns to the right and left of these blocks in a fused vertex are unchanged by the maps, so their contributions to the weight are preserved.
	
\end{proof}
	
	We include the following example to clarify how we check the color merging property in Figures~\ref{fig:llp1} and \ref{fig:llp2}. Consider the third row of Figure~\ref{fig:llp1}. The states in the ``before applying $p_i''$'' column are the following, where there are $m_k$ bosons of type $c_kd_{i+1}$ and $n_k$ bosons of type $c_kd_i$ in each monochrome column below. 
	
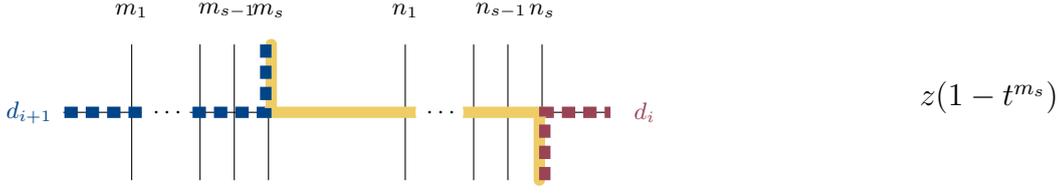
	\[
	\begin{array}{ccc}
		\input{figures/llp_dip1di_1} & \;\;\;\;\;\;\; & \raisebox{54pt}{$zt^{m_1 + m_2 + \cdots + m_s}$} \\
		\input{figures/llp_dip1di_2} && \raisebox{54pt}{$z(1-t^{m_1})t^{m_2 + m_3 + \cdots + m_s}$}  \\
		\input{figures/llp_dip1di_3} && \raisebox{54pt}{$z(1-t^{m_2})t^{m_3 + \cdots + m_s}$} \vspace{-20pt} \\
        \vdots &&  \vdots \vspace{5pt}\\
    \end{array}
	\]
    \[
    \begin{array}{ccc}
		\input{figures/llp_dip1di_4} && \raisebox{54pt}{\;\;\;\;\;\;\;\;\;\;\;\;\;\;\;\;\;\;\;\;\;\;\;$z(1-t^{m_s})$} \\
	\end{array}
	\]

	Taking the sum of the weights of these states, we have 
	\[
	zt^{m_1 + m_2 + \cdots + m_s} + z(1-t^{m_1})t^{m_2 + m_3 + \cdots + m_s} + z(1-t^{m_2})t^{m_3 + \cdots + m_s} + \cdots + z(1-t^{m_s}).
	\]
	Observing that the sum is telescoping, so it equals $z$. On the other hand, all of these states descend to the same state after applying $p_i''$
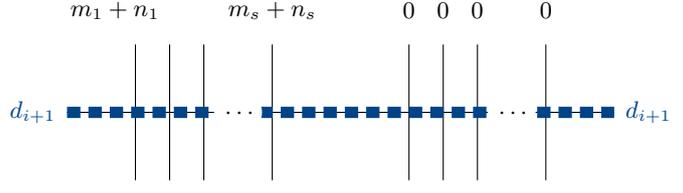
	\[
	\input{figures/llp_dip_ppp},
	 \]
	which has the same weight $z$.

\subsection{Global implications of color merging} In the previous subsection, we defined the maps $p_h$ and $p_v$ to send horizontal and vertical bicolored edges to edges of the colored model. We can define the map $p$ which sends a state $s$ of the bicolored model to $p(s)$ a state of the colored model. 

Let $\S(\lam,\wi{1},\wi{2},\wi{3},\wi{4},\z)$ be a system of the bicolored model. Define the system 
\[
\S^\text{bicol}_{\lam,\wi{4}}(\wi{1},\wi{3},\z) = \bigsqcup_{w_2} \S(\lam,\wi{1},\wi{2},\wi{3},\wi{4},\z)
\]
to be the set of states of the bicolored model with arbitrary left boundary condition. Let $\S^\text{col}_{\lam}(\wi{1},\wi{3},\z)$ be the system of the colored model as described in \cite{BN} with the specified boundary conditions, having $\wi{1}$ determine the ordering of the colors on the right boundary and $\wi{3}$ the ordring on the top boundary.

\begin{proposition}\label{prop:llp_global}
    Let $s_0$ be a state of the colored model $\S^\text{col}_\lam(\wi{1},\wi{3},\z)$. Let $\wi{4} \in S_r$ be a permutation. Then
    \[
    \beta^\text{col} (s_0) = \z^{N(1-r)}\sum_{\substack{s\in \S^\text{bicol}_{\lam,\wi{4}}(\wi{1},\wi{3},\z) \\ p(s) = s_0}} \beta^\text{bicol}(s),
    \]
    where the sum is over the collection of states with arbitrary left boundary condition and top boundary defined by $\wi{4}$ which lift from $s_0$.
\end{proposition}

\begin{proof}
    We follow the proofs of Proposition~5.2 of \cite{BN} and, more formally, Lemma~8.5 of \cite{BBBG}. We may construct all the liftings $s$ of the state $s_0$ to $\S^\text{bicol}_{\lam,\wi{4}}(\wi{1},\wi{3},\z)$ algorithmically by assigning colored spins to the edges left and below each vertex, visiting the vertices in order row by row from right to left. At the vertex under consideration at any given step in the algorithm, if the spins on the edges are labelled $a,b,c,d$ as in Proposition~\ref{prop:llp}, the vertices $b,c$ are known, while $a,d$ must satisfy $p_h(a) = A$ and $p_v(d)= D$, where $A,D$ are the corresponding spins of $s_0$. The result then follows from Proposition~\ref{prop:llp}. 
\end{proof}

Now define the scaled partition function of a bicolored system $\mathfrak{S}$ to be $\Zbar(\mathfrak{S}) = \z^{N(1-r)}Z(\mathfrak{S})$, where $N$ is the number of columns and $r$ is the number of dolors. The function $\Zbar$ is a polynomial based on Propositions~\ref{prop:llp} and \ref{prop:llp_global}, since the weights of states of the colored model are polynomial.

\begin{corollary}\label{cor:llp_col_pt_fn}
    Let $\wi{4}\in S_r$. Then
    \[
        Z(\S^\text{col}_\lam(\wi{1},\wi{3},\z)) = \sum_{\wi{2}} \Zbar(\S(\lam,\wi{1},\wi{2},\wi{3},\wi{4},\z)).
    \]
\end{corollary}

\begin{proof}
    This follows from Proposition~\ref{prop:llp_global} by summing over all states of $\S^\text{col}_\lam(\wi{1},\wi{3},\z).$
\end{proof}

\begin{corollary} \label{cor:llp}
    Let $\wi{3}, \wi{4}\in S_r$. Defining the uncolored models $\S^\text{uncol}_\lam(\z)$ as in Section 5 of \cite{BN},
    \[
        Z(\S^\text{uncol}_{\lam}(\z)) = \sum_{\wi{1}} \sum_{\wi{2}} \Zbar(\S(\lam,\wi{1},\wi{2},\wi{3},\wi{4},\z)).
    \]
\end{corollary}

\begin{proof}
    This follows from Corollary 5.3 of \cite{BN} and Corollary~\ref{cor:llp_col_pt_fn}.
\end{proof}

\section{Correspondence with Gelfand-Tsetlin patterns}\label{sec:gt_patterns}

Solvable lattice models correspond to many other families of combinatorial objects. The states of certain uncolored models, for example, correspond to Gelfand-Tsetlin patterns, semi-standard Young tableaux, alternating sign matrices \cite{Ku}, among many other objects. In this section, we will show that the bicolored bosonic vertex models in this paper correspond to certain 2-colored Gelfand-Tsetlin patterns, which we precisely define later. The main result is stated below, and an illustrative example is shown in Figure~\ref{fig:gt_bijection_example}.

\begin{theorem}\label{prop:gt_bijection}
	States of the lattice model and 2-colored Gelfand-Tsetlin patterns are in bijection.
\end{theorem}

\begin{figure}[!htbp]
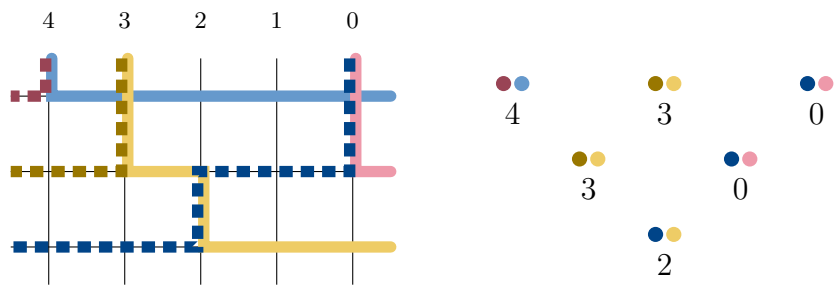

	\[
	\begin{array}{ccc}
		\input{figures/cdgtexlatticeone} &\;\;\;\;\; &
		\input{figures/cdgtexone}
	\end{array}
	\]
	\caption{A lattice model state and a 2-colored Gelfand-Tsetlin pattern hold the same information.}
	\label{fig:gt_bijection_example}
\end{figure}

\subsection{2-colored Gelfand-Tsetlin patterns}
A \emph{Gelfand-Tsetlin pattern} is a triangular array of numbers 
\[
\input{figures/gtgeneric},
\]
such that each row is a partition and the rows interleave.

We define a \emph{2-colored Gelfand-Tsetlin pattern} $G= (\Lambda,C,D)$ as a triangular array where each entry has 3 pieces of information: a number $\lam^{(i)}_j$, and two colors $c^{(i)}_j$ and $d^{(i)}_j$. The numbers $\lam^{(i)}_j$ form a Gelfand-Tsetlin pattern. The colors $c^{(i)}_j$ and $d^{(i)}_j$ are chosen from the sets $\c = \{c_1,c_2,\dots, c_n\}$ and $\d = \{d_1,d_2,\dots, d_m\}$ respectively according to the following axioms. 

\begin{enumerate}
	\item \(
	\displaystyle c^{(i-1)}_j \in \begin{cases}
		\{c^{(i)}_1, \dots, c^{(i)}_j\} & \text{if }\lam^{(i-1)}_j > \lam^{(i)}_{j+1} \\
		\{c^{(i)}_1, \dots, c^{(i)}_j,c^{(i)}_{j+1} \} & \text{if }\lam^{(i-1)}_j = \lam^{(i)}_{j+1} \\
	\end{cases}
	\)
	\item \(
	\displaystyle d^{(i-1)}_j \in \begin{cases}
		\{d^{(i)}_j, d^{(i)}_{j+1}, \dots, d^{(i)}_i\} & \text{if }\lam^{(i)}_j = \lam^{(i-1)}_{j} \\
		\{d^{(i)}_{j+1}, \dots, d^{(i)}_i \} & \text{if }\lam^{(i)}_j > \lam^{(i-1)}_{j} \\
	\end{cases}
	\)
	\item For each array $C,D$, a color in row $i$ appears at most as many times as it appears in row $i+1$.
	\item If $\lam^{(i)}_j = \lam^{(i)}_{j+1}$ then $(c^{(i)}_j, d^{(i)}_j) \leq (c^{(i)}_{j+1}, d^{(i)}_{j+1})$.
	\item If $\lam^{(i-1)}_j = \lam^{(i)}_{j+1}$ then $(c^{(i-1)}_j,d^{(i-1)}_j) \leq (c^{(i)}_{j+1},d^{(i)}_{j+1})$, and if $\lam^{(i)}_j = \lam^{(i-1)}_{j}$ then $(c^{(i)}_j,d^{(i)}_j) \leq (c^{(i-1)}_{j},d^{(i-1)}_{j})$.
	\item If $c^{(i)}_k$ appears in row $i$ but not in row $i-1$ of $C$, and respectively $d^{(i)}_l$ appears in row $i$ but not in row $i-1$ of $D$, then $l \leq k$.
\end{enumerate}

An example of a 2-colored Gelfand-Tsetlin pattern is shown on the right hand side of Figure~\ref{fig:gt_bijection_example}. In the figure, the colors are shown above their numbered entry, with those labelled by $C$ and $D$ to the right and left, respectively.

\subsection{Correspondence with the lattice model} 
In this section we show that 2-colored Gelfand-Tsetlin patterns correspond to states of the lattice model, proving Theorem~\ref{prop:gt_bijection}.

From a state of the lattice model we may construct a pattern $G = (\Lambda,C,D)$ as follows. Consider the bosons occupying the vertical edges just above row $i$ of the lattice model. Define the entries of row $\lambda^{(r-i+1)}$, $c^{(r-i+1)}$, $d^{(r-i+1)}$ as the column numbers, colors, and  dolors respectively associated to the bosons on these vertical edges above row $i$ as read from left to right. Now we argue that $G$ is a 2-colored Gelfand-Tsetlin pattern. By the color merging property and results that uncolored bosonic models correspond to Gelfand-Tsetlin patterns, $\Lambda$ constitutes a Gelfand-Tsetlin pattern. That the colors only may only travel downward and to the right shows that axiom (1) for 2-colored GT patterns is satisfied by $C$, and similarly that dolors travel down and to the left satisfies axiom (2). Axiom (3) is also clear by the interpretation that colors and dolors may be thought of as particles travelling along paths, and notably particles do not enter the model except at the top boundary. The monochrome ordering of bosons means that $G$ satisfies axiom (4), as two different bosons appearing in the same column must appear in the specified order. Axiom (5) is a consequence of the admissible types of fused vertices. Finally axiom (6) follows from the fact that a color and a dolor cannot both occupy the same horizontal edges. 

We construct a state of the lattice model in the corresponding way. A 2-colored Gelfand-Tsetlin pattern $G$ specifies the bosons occupying vertical edges, and  we will show that there is a unique way to fill in the horizontal edges which gives a state of the lattice model. In particular, we will show that every horizontal edge in a row of the lattice model must be occupied by exactly one color or dolor. Consider row $i$ of the lattice model. The coloring of the vertical edges above and below row $i$ are specified by rows $r-i+1$ and $r-i$ of $G$. We fill in the horizontal edges of row $i$ to connect the vertical lines for each color and dolor. Axiom (3) and the triangular shape of $G$ ensures that exactly one color and one dolor exit at row $i$ of the lattice model, and axiom (6) ensures that these particles do not cross as the color exits to the right and the dolor exits to the left. Axiom (1) ensures that the remaining colors travel to the right in the lattice model and axiom (2) ensures that dolors travel to the left. Finally, axioms (4) and (5) ensure that each fused vertex is admissible.

\end{document}